\documentclass[11pt]{amsart}

\usepackage[margin=1in]{geometry}
\usepackage{enumerate}
\usepackage{verbatim}
\usepackage{color}
\usepackage{todonotes}
\usepackage{tikz}
\usetikzlibrary{cd}
\usepackage{bbm}
\usepackage{adjustbox}
\usepackage[pagewise]{lineno}

\usepackage[pdftex,colorlinks,citecolor=blue,linkcolor=red]{hyperref}

\makeatletter
\@namedef{subjclassname@2020}{%
  \textup{2020} Mathematics Subject Classification}
\makeatother

\theoremstyle{plain}
\newtheorem{theorem}{Theorem}[section]
\newtheorem{lemma}[theorem]{Lemma}
\newtheorem{proposition}[theorem]{Proposition}
\newtheorem{corollary}[theorem]{Corollary}
\newtheorem{conjecture}[theorem]{Conjecture}

\theoremstyle{definition}
\newtheorem{definition}[theorem]{Definition}
\newtheorem{example}[theorem]{Example}
\newtheorem{remark}[theorem]{Remark}

\newcommand{\Bbbk}{\mathbbm{k}}

\DeclareMathOperator{\im}{im}

\newcommand{\R}{\mathbb{R}}

\newcommand{\cI}{\mathcal{I}}
\newcommand{\cB}{\mathcal{B}}

\newcommand{\bp}{\mathbf{p}}
\newcommand{\bq}{\mathbf{q}}

\newcommand{\bb}{\mathbf{b}}
\newcommand{\bI}{\mathbf{I}}

\newcommand{\si}{\mathsf{i}}
\newcommand{\sj}{\mathsf{j}}
\newcommand{\sk}{\mathsf{k}}
\newcommand{\sfa}{\mathsf{a}}
\newcommand{\sfb}{\mathsf{b}}
\newcommand{\sfc}{\mathsf{c}}
\newcommand{\sfd}{\mathsf{d}}

\newcommand{\co}{\colon}
\newcommand{\ra}{\rightarrow}
\newcommand{\lang}{\langle}
\newcommand{\rang}{\rangle}

\newcommand{\bid}{\mathbf{1}}

\newcommand{\lcm}[1]{\mathrm{lcm}({#1})}

\begin{document}
\title[Splines on graphs and polynomial splines on cycles]{Generalized splines on graphs with two labels and polynomial splines on cycles}

\author{Portia Anderson}
\address{Department of Mathematics, Cornell University, Ithaca, NY.}
\email{pxa2@cornell.edu}
\author{Jacob P. Matherne}
\address{Mathematical Institute, University of Bonn, Bonn, Germany and Max Planck Institute for Mathematics, Bonn, Germany.}
\email{jacobm@math.uni-bonn.de}
\author{Julianna Tymoczko}
\address{Department of Mathematics and Statistics, Smith College, Northampton, MA.}
\email{jtymoczko@smith.edu}

\begin{abstract}

{\em Generalized splines} are an algebraic combinatorial framework that generalizes and unifies various established concepts across different fields, most notably the classical notion of splines and the topological notion of GKM theory.  The former consists of piecewise polynomials on a combinatorial geometric object like a polytope, whose polynomial pieces agree to a specified degree of differentiability. The latter is a graph-theoretic construction of torus-equivariant cohomology that Shareshian and Wachs used to reformulate the well-known Stanley--Stembridge conjecture, a reformulation that was recently proven to hold by Brosnan and Chow and independently Guay-Paquet. 

This paper focuses on the theory of generalized splines.  A generalized spline on a graph $G$ with each edge labeled by an ideal in a ring $R$ consists of a vertex-labeling by elements of $R$ so that the labels on adjacent vertices $u, v$ differ by an element of the ideal associated to the edge $uv$.  We study the $R$-module of generalized splines and produce minimum generating sets for several families of graphs and edge-labelings: $1)$ for all graphs when the set of possible edge-labelings consists of at most two finitely-generated ideals, and $2)$ for cycles when the set of possible edge-labelings consists of principal ideals generated by elements of the form $(ax+by)^2$ in the polynomial ring $\mathbb{C}[x,y]$.  We obtain the generators using a constructive algorithm that is suitable for computer implementation and give several applications, including contextualizing several results in the theory of classical (analytic) splines.
\end{abstract}

\thanks{Portia Anderson received support from the Smith College Center for Women in Mathematics and from NSF-MTCP-1143716. Jacob Matherne received support from NSF-DMS-1638352, the Association of Members of the Institute for Advanced Study, the Hausdorff Research Institute for Mathematics in Bonn, and the Deutsche Forschungsgemeinschaft (DFG) under Germany's Excellence Strategy - GZ 2047/1, Projekt-ID 390685813.  Julianna Tymoczko received support from NSF-DMS-1362855 and NSF-DMS-1800773.}

\subjclass[2020]{05C25 (Primary); 05C78, 05E16, 41A15 (Secondary)}

\keywords{generalized splines, minimum generating sets}

\maketitle

\section{Introduction}
\label{sec:intro}

Splines are a fundamental tool in applied mathematics and analysis, used in fields from data interpolation \cite{dB01} to computer graphics and design \cite{BBB87}.  Classically, they are defined as piecewise polynomials on a combinatorial partition of a geometric object that agree up to some specified differentiability on the intersection of the top-dimensional pieces of the partition.  The most common example of these combinatorial partitions in the literature is a polyhedral or simplicial decomposition of a suitable region in Euclidean space.

Splines also appear in algebraic topology under the name {\em GKM theory}.  GKM theory is a graph-theoretic construction of torus-equivariant cohomology developed by Goresky, Kottwitz, and MacPherson \cite{GKM98}.  It has had particular impact in algebraic combinatorics, especially Schubert calculus \cite{KT03, T16}.  More recently, Shareshian and Wachs conjectured that GKM theory could be used to reframe what's called the \emph{Stanley--Stembridge conjecture} in combinatorial representation theory as a conjecture about a symmetric group action on the torus-equivariant cohomology of a particular family of varieties \cite{SW16}.  Their geometric interpretation of the  Stanley--Stembridge conjecture---though not the conjecture itself---was recently proven to hold by Brosnan and Chow and, independently, Guay-Paquet \cite{BC18, GP13}.  This has led to an explosion of work relying on properties of splines in the GKM context \cite{AHM19, HPT21, HP18}.

This paper considers a simultaneous algebraic generalization of both classical splines and the splines occurring in GKM theory: given a (combinatorial) graph $G$ with each edge labeled by an ideal in some fixed ring $R$, a generalized spline is an $R$-labeling of the vertices so that the labels on adjacent vertices $u, v$ differ by an element of the ideal labeling the edge $uv$.  This formulation is due to work of the third author with Gilbert and Viel \cite{GTV16}, but was first used by Billera \cite{Bil88} and (in the context of equivariant cohomology) by Guillemin--Zara  \cite{GuiZar00, GuiZar01d, GuiZar01}.  The construction of generalized splines is essentially dual to the classical definition of splines \cite[Theorem 2.4]{Bil88} (equivalently, Proposition~\ref{proposition: billera result}). For example, in the case of a triangulation of a region in the plane, the vertices of $G$ correspond to triangles of the triangulation, and the edge-relations correspond to differentiability conditions across intersections of triangles.  Generalized splines coincide with the typical construction of GKM theory if the graph, the ring, and the ideals all satisfy very particular conditions \cite{GKM98}.  (See Section~\ref{sec:applications} for more.)

One of the most important problems in the study of splines is to identify the size of the spline space, interpreted either as the dimension of the vector space of classical splines of degree at most $d$ \cite{AlfSch87, AlfSch90, Hon91, SSY, SchSti02, Sch79, StiYua19, Str74} (see \cite{LaiSch07} for a survey in the bivariate case) or as the (minimum) number of generators of the module of generalized splines \cite{AltSar19,AltSar19b, BilRos91, DiP12, GuiZar01, GuiZar03, BHKR15, GTV16, ACFGMT20}. 

In this paper, we compute the minimal number of generators of the module of generalized splines over several families of graphs for different collections of rings.  Our most general result is Theorem~\ref{thm:2labalg}, which gives an algorithm using graph connectivity to compute a minimum set of generators for the module of generalized splines over any graph $G$ whose set of edge-labels consists of at most two distinct ideals.  The only hypothesis on $R$ is that it be a unique factorization domain (UFD).  Theorem~\ref{thm:onelabel} specializes to the case when all edges of $G$ are labeled by the same ideal; in that case, if the ideal is principal then the module of generalized splines is free over $R$ and its rank is precisely the number of vertices in $G$.

We then specialize $R$ to be a polynomial ring, typically using the assumption that the edge-labels are principal ideals generated by homogeneous polynomials of the same degree. These assumptions may seem restrictive but are not.  Indeed, in both classical splines and GKM theory, splines use polynomial rings as their base ring; furthermore, all known applications use principal ideals as edge-labels.  (See Remark~\ref{rem:principalisok} and Sections~\ref{section: GKM summary} and~\ref{section: classical spline summary} for more extensive discussion.) Moreover, the edge-labels used in GKM theory arise as the weights of torus actions on a geometric space, and are naturally homogeneous. Even in cases when the edge-labels are not a priori homogeneous (as in classical splines), we can homogenize. (See Section~\ref{sec:homog} for more details about homogenization.)  Corollary~\ref{corollary: homogeneity means quotient works} proves that homogenization induces a natural vector space isomorphism between the classical vector space of splines of degree at most $d$ and the module of generalized splines over the polynomial quotient 
\[
\mathbb{R}[x_1, x_2, \ldots, x_n]/\lang \textup{all monomials in the } x_i \textup{ of degree at least } d+1 \rang
\] 
considered as a real vector space. Corollary~\ref{corollary: homogeneity means quotient works} is an application of the general framework of quotient splines that we develop in Section~\ref{sec: quotient splines}, together with Billera's result. 

Classical splines do not naturally form a ring since multiplication generally increases degree.  This is in contrast to the setting of generalized splines, where the module of generalized splines associated to a fixed ring and edge-labeled graph  actually forms a ring.  (See the discussion after Remark~\ref{remark:gkm}.)     However, identifying the vector space of classical splines with the elements of this quotient space allows us to consider a ring structure on splines. In this sense, Corollary~\ref{corollary: homogeneity means quotient works} {\em provides a new algebraic tool} for classical splines.

Under the assumptions of the previous paragraph, we prove one other main result. Theorem~\ref{thm:polysplines} computes explicit homogeneous generators for all generalized splines on cycles whose edges are labeled by polynomials of the form $(x+ay)^2$
and shows that these generalized splines cannot be obtained from fewer generators. Indeed, it shows that these generators form a basis. This is a remarkably uniform result that depends only on the number of distinct edge-labels, and not the underlying geometry.

\begingroup
\renewcommand\thetheorem{\ref{cor:polysplines}}
\begin{corollary}
Suppose $C_n$ is a cycle with $n$ vertices and that each edge is labeled by a principal ideal generated by a polynomial of the form $(x+ay)^2$.  Then the module of generalized splines has a basis of the following form:
\begin{itemize}
    \item If there is only one distinct edge-label: one homogeneous generator of degree zero and $n-1$ of degree two.
    \item If there are two distinct edge-labels: one homogeneous generator of degree zero, $n-2$ of degree two, and one of degree four.
    \item If there are at least three distinct edge-labels: one homogeneous generator of degree zero, $n-3$ of degree two, and two of degree three.
\end{itemize}
\end{corollary}
\endgroup
Section~\ref{sec:apps} focuses on applications related to the {\em lower bound conjecture} in the classical theory of bivariate splines.  The lower bound conjecture arises from an explicit polynomial in $r,d$ constructed from a triangulation of a given region $\Delta$ of the plane. Strang conjectured that this polynomial computes the dimension of $S^r_d(\Delta)$ for specific families of $r$, $d$, and $\Delta$ \cite{Str74}.  Schumaker showed that in fact the polynomial is a lower bound for all $r,d,\Delta$ \cite{Sch79}.  Considerable work has happened on this problem since: Alfeld and Schumaker showed that the polynomial gives the dimension when $d \geq 4r+1$ \cite{AlfSch87}, which Hong later tightened to $d \geq 3r+2$ \cite{Hon91}; at the same time, Billera proved Strang's conjecture for $r=1$ and $d=3$ as long as the triangulation $\Delta$ is generic \cite{Bil88}.  

Our last two results provide a theoretical foundation contextualizing the lower bound formula when $r=1$ and $d=3$.  In this case, most mathematicians believe the formula actually computes the dimension of $S^1_3(\Delta)$; there are no known counterexamples to this claim despite significant and ongoing efforts \cite{SchSti02,StiYua19,SSY}.   (See \cite[Chapter 9]{LaiSch07} for more history and context.) 

We show:
\begin{itemize}
    \item Theorem~\ref{thm:polysplines}, Lemma~\ref{lem:dualcycles}, and Corollary~\ref{corollary: dimension of splines on pinwheels} together give an alternative proof of Schumaker's characterization of classical splines on a single interior cell, namely the ``pinwheel triangulations" consisting of a single interior vertex and a number of triangles incident to that vertex and covering a small neighborhood around that vertex \cite[Theorems 9.3 and 9.12]{LaiSch07}.
    \item When $r=1$ and $d=3$, the lower bound formula consists of terms contributed by boundary and interior vertices of the triangulation, a correction term for certain interior vertices called ``singular vertices," and a constant term from polynomials defined on the entire triangulation (not piecewise).  Corollary~\ref{cor:final} explains the correction term accounting for ``singular vertices" as the unique geometrically realizable triangulations that correspond to cycles with exactly two distinct edge-labels.
\end{itemize} 

\subsection*{Acknowledgments}

Michael DiPasquale's assistance with Macaulay2 \cite{M2} was invaluable to the data collection that seeded this paper. The second author would like to thank Cleo Roberts and Claudia Yun for helpful discussions about splines.  The authors are also grateful to the anonymous referee; their comments greatly improved the article and also encouraged us to   write Section~\ref{sec: quotient splines}. 

\section{Generalized splines on graphs}
\label{sec:splines}

This section reviews basic definitions and constructions, including terminology from graph theory and essential results about generalized splines. We state most results in this paper for generalized splines on connected graphs because splines for arbitrary graphs can be obtained from generalized splines on the connected components via direct sum (see Proposition \ref{proposition: direct sum of connected graphs}).  

\subsection{Graphs}
\label{sec:graphs}

For a graph $G = (V,E)$, we denote its (finite) set of vertices by $V$ and its (finite) set of edges by $E$.  We write elements of $E$ as pairs of distinct vertices; for example, $e=uv$ is the edge that joins vertex $u$ and vertex $v$, and we say $u$ and $v$ are adjacent. (Note that $uv = vu$ since edges are unoriented.)  
Graphs in this paper have at most one edge between any given pair of vertices.  

If $G' = (V',E')$ is another graph such that $V' \subseteq V$ and $E' \subseteq E$, then $G'$ is called a {\em subgraph} of $G$.  The {\em induced subgraph} $G[V']$ of $V'$ is the graph with vertex set $V'$ and edge set consisting of all edges in $E$ with both vertices in $V'$.  The {\em neighborhood} $N_G(V')$ of $V'$ is the set of vertices in $V$ that are adjacent to at least one vertex in $V'$.  We also define the graph $G - E' := (V,E\setminus E')$.  

A {\em path} in $G$ is a finite sequence of edges $(u_1u_2, u_2u_3, \ldots, u_{n-2}u_{n-1}, u_{n-1}u_n)$, with each $u_i \in V$, such that each pair of successive edges shares a vertex.  A {\em connected component} of $G$ is a subgraph $G'$ of $G$ with the property that any two vertices of $G'$ are joined by a path lying entirely in $G'$.  If $G$ has exactly one connected component, then $G$ is a {\em connected graph}.

\begin{proposition}\label{prop:order}
If $G = (V,E)$ is a connected graph, then there is an ordering $v_1, \ldots, v_{|V|}$ on $V$ such that for every $1 < i \le |V|$ the vertex $v_i$ is adjacent to at least one of the vertices $v_1, v_2, \ldots, v_{i-1}$.   
\end{proposition}

\begin{proof}
We proceed by induction on the number of vertices currently ordered.  For the base case, arbitrarily choose a first vertex $v_1 \in V$. Assume as inductive hypothesis that we have ordered $v_1, \ldots, v_k$ for some $1 < k < |V|$ so that $v_i$ is adjacent to at least one of the vertices $v_1, v_2, \ldots, v_{i-1}$ for each $i \leq k$.  Suppose that $k < |V|$ and that the neighborhood $N_G(\{v_1, \ldots, v_k\})$ is contained in the set $\{v_1, \ldots, v_k\}$.  Then no vertex in $V - \{v_1, \ldots, v_k\}$  shares an edge with any vertex in $\{v_1, \ldots, v_k\}$.  This means that $G$ is disconnected, contradicting our hypothesis on $G$.   Thus if $|V| \geq k+1$ there is some $v_{k+1} \in N_G(\{v_1, \ldots, v_k\})$, and the claim holds by induction.
\end{proof}

\subsection{Generalized splines and minimum generating sets}
\label{sec:mingen}

Let $R$ be a commutative ring with identity denoted by $1$.  (We will add more conditions on $R$ as they become necessary.)  Let $\cI$ be the set of ideals of $R$.  A function $\alpha \co E \ra \cI$ is called an {\em edge-labeling} of $G$.   We write $(G,\alpha)$ to mean a graph together with an edge-labeling, and call it an {\em edge-labeled graph}. If $|\alpha(E)| \leq k$, then we call $(G,\alpha)$ a {\em $k$-labeled graph}.  Note that, we have suppressed explicit mention of the vertex set and edge set in the notation of an edge-labeled graph, but these will always be clear from context.  

\begin{definition}\label{def:generalizedspline}
Let $(G,\alpha)$ be an edge-labeled graph.  A \emph{generalized spline} on $(G,\alpha)$ is a vertex-labeling $\bp \in \bigoplus_{v \in V} R$ that satisfies the GKM condition:
\[\begin{array}{c}
\text{for every edge } e = uv \in E, 
\textup{ the difference between } \\ \textup{ coordinates at the endpoints is } \bp_u - \bp_v \in \alpha(e).
\end{array}
\]
\end{definition}

For the remainder of the paper, we write ``spline" to mean ``generalized spline", unless we specifically write ``classical spline".  We also note that the definition of spline makes sense for noncommutative rings $R$, but, to our knowledge, this has not been explored deeply in the literature.

We sometimes write a spline $\bp = (\bp_{v_1}, \ldots, \bp_{v_{|V|}})$ as a $|V|$-tuple when we have a particular ordering on $V$ in mind.  See Figure~\ref{fig:splineexample} for an example of a spline on an edge-labeled graph.

\begin{remark}\label{remark:gkm}
The name ``GKM condition" (also appearing in \cite{GTV16}) refers to work by Goresky, Kottwitz, and MacPherson, where this condition appears while combinatorially computing the equivariant cohomology of certain varieties carrying well-behaved torus actions \cite{GKM98}.
\end{remark}

We write $R_{G,\alpha}$ for the set of splines on the edge-labeled graph $(G,\alpha)$.  It is well known that $R_{G,\alpha}$ is itself a ring with identity \cite[Proposition 2.4]{GTV16}.  The unit $\mathbf{1} \in R_{G,\alpha}$ is given by $\mathbf{1}_v = 1$ for all vertices $v$.  (We call $\mathbf{1}$ the \emph{trivial spline}.)  Addition and multiplication in $R_{G,\alpha}$ are defined pointwise; that is, $(\bp + \bq)_v = \bp_v + \bq_v$ and $(\bp \bq)_v = \bp_v \bq_v$ for all $v \in V$.  Moreover, $R_{G,\alpha}$ carries the structure of an $R$-module given by $r \cdot \bp = (r \bp_v)_{v \in V}$ for any $r \in R$.

The following proposition from \cite{GTV16} confirms that our results extend from connected graphs to arbitrary graphs.  Recall if $G' = (V', E')$ and $G'' = (V'', E'')$ are graphs, then their \emph{union} is defined as 
\[G' \cup G'' = (V' \cup V'', E' \cup E'').\]    

\begin{proposition}[{\cite[Proposition 2.11]{GTV16}}] \label{proposition: direct sum of connected graphs}
Let $(G', \alpha')$ and $(G'', \alpha'')$ be two disjoint edge-labeled graphs, namely $V' \cap V'' = \emptyset$ and $E' \cap E'' = \emptyset$.  If $G = G' \cup G''$ and $\alpha$ is the edge-labeling on $G$ defined by restricting to $\alpha'$ on $G'$ and $\alpha''$ on $G''$, then $R_{G,\alpha} = R_{G',\alpha'} \oplus R_{G'',\alpha''}$.
\end{proposition}

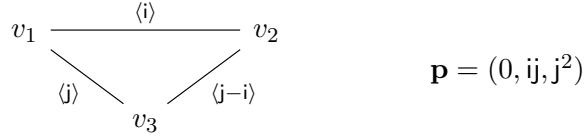
\begin{figure}
\begin{tikzcd}
v_1 \ar[rr,dash,"\lang\si\rang"] & & v_2 \\
 & v_3 \ar[ul,dash,"\lang\sj\rang"] \ar[ur,dash,swap,"\lang\sj-\si\rang"] & 
\end{tikzcd}
\qquad \qquad
\begin{tikzcd}
\bp = (0,\si\sj,\sj^2)
\end{tikzcd}
\caption{An edge-labeled complete graph on three vertices together with a spline $\bp$ on it. Here $\si,\sj \in R$.}
\label{fig:splineexample}
\end{figure}

In this paper, we present algorithms for producing minimum generating sets for $R_{G,\alpha}$ for a variety of edge-labeled graphs $(G,\alpha)$.

\begin{definition}
A \emph{generating set} $\cB$ for $R_{G,\alpha}$ is a set of splines in $R_{G,\alpha}$ which generates $R_{G,\alpha}$ as an $R$-module.  The set $\cB$ is called a \emph{minimum generating set} (MGS) if it is a generating set with the property that no other generating set has fewer elements than $\cB$.
\end{definition}

The general question of when $R_{G,\alpha}$ is a free $R$-module is complicated.  In topological applications, $R_{G,\alpha}$ is typically assumed to be free---this is the main implication of {\em equivariant formality}, which is one of the hypotheses in the machinery of GKM theory (see \cite{GKM98,Tym05} for more details).  In analytic applications, they need not be (see, for example, \cite{DiP12}).  For the most part we do not address this question, though the following lemma applies to several of our results.

\begin{lemma}\label{lemma: upper-triangular basis}
Let $R$ be an integral domain and $(G,\alpha)$ be an edge-labeled graph.  If $\mathcal{B}$ is an MGS for $R_{G,\alpha}$ that is triangular\footnote{An MGS $\cB$ is \emph{(upper or lower) triangular} with respect to a vertex ordering $v_1,\ldots,v_{|V|}$ on $V$ if, after ordering the entries of the elements of $\cB$ according to the ordering on $V$, the matrix whose columns are the elements of $\cB$ is a (upper or lower) triangular matrix with nonzero diagonal entries.  
} with respect to some vertex ordering $v_1, \ldots, v_{|V|}$ on $V$, then $R_{G,\alpha}$ is a free $R$-module with basis $\cB$.
\end{lemma}

\begin{proof}
Since the MGS $\mathcal{B}$ is triangular, it has at most $|V|$ elements and we may order the basis elements $\mathcal{B} = \{\bb^{i_1}, \bb^{i_2}, \ldots, \bb^{i_{k'}}\}$ so that $\bb^i_{v_j} = 0$ for all $j<i$ and $\bb^i_{v_i} \neq 0$ for all $i \in \{i_1,i_2,\ldots,i_{k'}\}$.  

Now suppose $\sum_{\bb^i \in \mathcal{B}} c_i \bb^i = 0$ is a linear dependence.  We prove by induction on $i$ that all $c_i$ are zero---the base case, that $c_{i_1} = 0$, is clear by triangularity.  If $c_i=0$ for all $i < i_0$, then we have
\[
\sum_{\bb^i \in \mathcal{B}} c_i \bb^i_{v_{i_0}} = c_{i_0} \bb^{i_0}_{v_{i_0}},
\]
since all $c_i$ with $i<i_0$ are zero by the inductive hypothesis and all $\bb^i_{v_{i_0}}$ with $i > i_0$ are zero by  triangularity.  We assumed the displayed expression was zero, so $c_{i_0} \bb^{i_0}_{v_{i_0}}=0$.  But $\bb^{i_0}_{v_{i_0}}$ is nonzero by assumption and $R$ is an integral domain, so $c_{i_0}$ is zero.  The claim follows.
\end{proof}

The next result gives a lower bound on the number of elements of an MGS.  (This lower bound does not hold if $R$ has zero divisors \cite{BowTym15}.)

\begin{lemma}\label{lem:folk}
If $R$ is an integral domain, then the number of elements of an MGS $\cB$ is at least $|V|$.
\end{lemma}

\begin{proof}
If $R$ is an integral domain and $(G,\alpha)$ is a connected edge-labeled graph, then the module of splines $R_{G,\alpha}$ contains a free $R$-submodule $M$ generated by $|V|$ elements \cite[Corollary 5.2]{GTV16}.  Now consider the image of the $R$-modules $R_{G,\alpha} \supseteq M$ under the map induced by including $R$ into its field of fractions. The image of $M$ is a vector space of dimension $|V|$ and the image of $R_{G,\alpha}$ is a vector space that both contains the image of $M$ and is generated by the image of $\cB$.  Thus there are at least $|V|$ elements in $\cB$.
\end{proof}

\begin{remark}\label{rem:principalisok}
Many of our key results apply to edge-labelings by finitely-generated ideals.  However, our results treat principal ideals.  We do this for two reasons.  First, most applications of splines use edge-labels that are principal ideals (see Section~\ref{sec:applications} for more).  Second, our arguments usually generalize easily to finitely-generated ideals.  Indeed, the main step of many of our arguments uses triangular MGSs in which each generator also satisfies ${\bf b}^i_u \in \{0,r\}$ for all $u\in V$ and some fixed ring element $r$.  In this context, it is straightforward to extend the main results of this paper from edge-labels that are principal ideals to edge-labels that are finitely generated: instead of creating a single spline ${\bf b}^i$ for which ${\bf b}^i_{v_i}$ generates the principal ideal associated to an edge incident to $v_i$, we create a set of generators $\{{\bf b}^{i,1}, {\bf b}^{i,2},\ldots, {\bf b}^{i,k}\}$ that minimally generate the ideal labeling that edge. (This kind of argument has been used previously in the literature \cite[Propositions 2.4 and 2.6]{HarTym17}.) Expanding the generator set in this fashion gives analogous versions for edge-labelings with finitely-generated ideals of Theorem~\ref{thm:onelabel}, Theorem~\ref{thm:2labalg}, and the dimension computations in Corollary~\ref{cor:repeatedge}.
\end{remark}

\section{Algorithm to produce an MGS on $2$-labeled graphs}
\label{sec:algorithm}

Recall that a $k$-labeled graph is an edge-labeled graph $(G,\alpha)$ such that $|\alpha(E)| \le k$; i.e., the set of edge-labels used consists of at most $k$ distinct ideals.  In this section, we give an algorithm to produce an MGS for an arbitrary connected $2$-labeled graph $(G,\alpha)$.  As a warm-up, in Section~\ref{sec:onelabel}, we treat the $1$-labeled case; in Section~\ref{sec:twolabels}, we treat the $2$-labeled case. Throughout this section, $G$ denotes an arbitrary connected graph.

\subsection{One edge-label}
\label{sec:onelabel}
Let $\alpha \co E \ra \cI$ be a constant edge-labeling function; that is, the image of $\alpha$ consists of a single principal ideal $I = \lang \si \rang$.  For a given $v \in V$, denote by $\bI^v$ the \emph{indicator spline} of the ideal $I$ at the vertex $v$.  In other words, $\bI^v$ is the spline with $\bI^v_v = \si$ and $\bI^v_u = 0$ for all $u \ne v$.

\begin{theorem}\label{thm:onelabel}
Fix an ordering $v_1,\ldots,v_{|V|}$ on $V$ as in Proposition~\ref{prop:order}, and let $\alpha \co E \ra \cI$ be the constant edge-labeling $\alpha(e) = I = \lang \si \rang \in \cI$ for all $e \in E$.  

Then the set $\cB = \{\mathbf{1},\bI^{v_2},\ldots,\bI^{v_{|V|}}\}$ is an MGS for $R_{G,\alpha}$.  Moreover, if $R$ is an integral domain then $R_{G,\alpha}$ is a free $R$-module with basis $\cB$.
\end{theorem}

\begin{proof}
 Let $\bp \in R_{G,\alpha}$ be an arbitrary spline.  We claim that there exist $r_2,\ldots,r_{|V|} \in R$ such that
\begin{equation}\label{eqn:onelabel}
\bp = \bp_{v_1}\mathbf{1} + r_2 \bI^{v_2} + \cdots + r_{|V|} \bI^{v_{|V|}}.
\end{equation}
We will prove the following statement, which is equivalent to Equation \eqref{eqn:onelabel}: for every $2 \le j \le |V|$, there exists $r_j \in R$ such that $\bp_{v_j} - \bp_{v_1} = r_j \si$. We proceed by induction on $j$.

When $j=2$, Proposition~\ref{prop:order} ensures that $v_2$ is adjacent to $v_1$.   Thus there exists $r_2 \in R$ such that $\bp_{v_2} - \bp_{v_1} = r_2 \si \in \alpha(v_2v_1)$, as desired. 

Our inductive hypothesis states: if $j \in \{2, \ldots, |V|-1\}$, then for all $k$ with $2 \le k \le j$ there exists $r_k \in R$ such that $\bp_{v_k} - \bp_{v_1} = r_k \si$.  By Proposition~\ref{prop:order}, the vertex $v_{j+1}$ is adjacent to some $v_{\ell}$ with $1 \le \ell \le j$.  If $v_{j+1}$ is adjacent to $v_1$, then the GKM condition ensures that there exists $r_{j+1} \in R$ with $\bp_{v_{j+1}} - \bp_{v_1} = r_{j+1} \si$ as desired.  Otherwise, the spline $\bp - \bp_{v_1}\mathbf{1} - r_{\ell} \bI^{v_{\ell}}$ satisfies 
\[(\bp - \bp_{v_1}\mathbf{1} - r_{\ell} \bI^{v_{\ell}})_{u} = \left\{ \begin{array}{ll}
0 & \textup{ when } u = v_{\ell}, \\
\bp_{v_i} - \bp_{v_1} & \textup{ when } u = v_i \textup{ for } i\neq \ell. \end{array} \right.\]  
The GKM condition when $u=v_{j+1}$ implies that there is some $r_{j+1} \in R$ such that $\bp_{v_{j+1}} - \bp_{v_1} = r_{j+1} \si \in \alpha(v_{j+1}v_\ell)$.  Equation~\eqref{eqn:onelabel} follows by induction, so $\cB$ is an MGS for $R_{G,\alpha}$ by Lemma~\ref{lem:folk}.

Finally, if $R$ is an integral domain then Lemma~\ref{lemma: upper-triangular basis} applies, proving that $R_{G,\alpha}$ is free with basis $\cB$.
\end{proof}

\subsection{Two edge-labels}
\label{sec:twolabels}

Now suppose the edge-labeling $\alpha\co E \ra \cI$ has image $\{I,J\} \subseteq \cI$ with $I = \lang \si \rang$ and $J = \lang \sj \rang$.  The theorem below gives an algorithm for producing an MGS for $R_{G,\alpha}$.  The basic idea of the proof is to consider the neighbors of each vertex $v_i$ successively.  If $v_i$ is connected to the first $i-1$ vertices only through paths with a single edge-label, then we can find a generator that uses only that edge-label; otherwise, we need a generator that is an indicator spline with nonzero entry given by the product of the two edge-labels.

We start with a graph-theoretic lemma.  Recall that a \textit{simple} path is a path in which no vertices are repeated.

\begin{lemma} \label{lemma: 2-label cut-sets}
Let $(G,\alpha)$ be a connected edge-labeled graph with edge-labeling $\alpha\co E \rightarrow \mathcal{I}$ such that $|\alpha(E)| = 2$.  Fix a vertex $u$ and a subset of vertices $S \subseteq V - \{u\}$ with nonempty intersection $N_G(\{u\}) \cap S$. Then the following are equivalent: 
\begin{enumerate}
    \item Choose any $u' \in S$.  Let
    \[G' = G - \{vw \in E \mid \alpha(vw) = \alpha(uu')\}\]
    and let $C=(V',E')$ be the connected component of $G'$ containing $u$.  Then $V' \cap S$ is empty.
    \item Every simple path from $u$ to a vertex in $S$ has at least one edge labeled $\alpha(uu')$.
\end{enumerate}
\end{lemma}

\begin{proof}
Let $\alpha(uu')=I$ and write $\alpha(E) = \{I, J\}$.  The edges in $G'$ are precisely the edges in $G$ that are labeled $J$.  The vertex $u$ is connected to a vertex $u'$ in the graph $G'$ if and only if there is a path between $u$ and $u'$ in $G'$ --- in other words, if and only if there is a path from $u$ to $u'$ whose edges are all labeled $J$.  It follows that Condition (1) fails if and only if Condition (2) fails, as desired.
\end{proof}

We now prove our claim.

\begin{theorem}\label{thm:2labalg}
Let $R$ be a UFD.  Let $(G,\alpha)$ be a connected edge-labeled graph with edge-labeling $\alpha\co E \ra \cI$ having image $\{\lang\si\rang,\lang\sj\rang\}$. Choose an ordering on $V$ as in Proposition~\ref{prop:order}.  For every $1 < i \le |V|$, define the spline $\bb^{i}$ as follows.
Choose some $v_j\in N_G(\{v_i\})$ with $j < i$, and write $\lang\sk\rang := \alpha(v_iv_j)$.  Let 
\[G' = G - \{uv \in E \mid \alpha(uv)=\lang \sk \rang\},\] 
and let $C = (V',E')$ be the connected component of $G'$ containing $v_i$.  
\begin{enumerate}
\item[(a)] If $V' \subseteq \{v_i,v_{i+1},\ldots,v_{|V|}\}$, then set $\bb^{i}_u = \sk$ for all $u \in V'$ and $\bb^{i}_u = 0$ for all $u \not\in V'$.
\item[(b)] Otherwise, set $\bb^{i}_{v_i} = \lcm{\si,\sj}$ and $\bb^{i}_u = 0$ for all $u \ne v_i$.
\end{enumerate}
Each $\bb^i$ depends only on the ordering of the vertices in $V$ and not on the choice of vertex $v_j$ in $N_G(\{v_i\})$ used to define $\bb^i$. Moreover $R_{G,\alpha}$ is a free $R$-module, and the set $\cB = \{\mathbf{1}, \bb^{2}, \ldots, \bb^{{|V|}}\}$ is a basis for $R_{G,\alpha}$.  
\end{theorem}

\begin{remark}
Note that if the image of $\alpha$ is a single edge-label, then $C$ always consists of the single vertex $v_i$ and Case (b) never applies.  Thus, Theorem~\ref{thm:onelabel} is a special case of Theorem~\ref{thm:2labalg}.
\end{remark}

\begin{proof}
We prove that $\cB$ is an MGS for $R_{G,\alpha}$.  Since $R$ is an integral domain, Lemma~\ref{lemma: upper-triangular basis} then implies the claim.

First we show the definition of $\bb^i$ depends only on the order of the vertices and not on the choice of $v_j \in N_G(\{v_i\})$.  Indeed suppose we chose $v_j \in N_G(\{v_i\})$ for which Case (a) holds.  An edge is an example of a simple path; applying Lemma~\ref{lemma: 2-label cut-sets} allows us to conclude that $\alpha(v_i v_{j'}) = \lang \sk \rang$ for every edge $v_i v_{j'}$ with $j' < i$.  Thus if one choice of $v_j$ leads to Case (a) then any other choice of $\{v_1,\ldots,v_{i-1}\} \cap N_G(\{v_i\})$ gives rise to the same edge-label $\lang k \rang$ and hence the same graph $G'$ so also Case (a) and the same $\bb^i$.  Otherwise, all choices of $v_j \in N_G(\{v_i\})$ give rise to Case (b), for which $\bb^i$ is defined independent of $v_j$.

Next we confirm that for all $i > 1$ the function $\bb^{i}$ is a spline. There are two cases.
\begin{enumerate}
    \item[(a)] If $\bb^{i}$ was produced by Case (a), then for every edge $uw \in E$
\[
\bb^{i}_u - \bb^{i}_w = \left\{\begin{array}{ll}\sk - \sk = 0 & \text{if both } u,w \textup{ are in } V',\\ 0 - 0 = 0 & \text{if neither } u,w \textup{ are in } V',\\ \pm(\sk - 0) = \pm \sk & \text{if exactly one of } u, w \textup{ are in } V'. \end{array}\right.
\]
When the difference $\bb^{i}_u - \bb^{i}_w = 0$ the function $\bb^{i}$ satisfies the GKM condition at edge $uw$ trivially so the GKM condition is satisfied.  If exactly one of $u, w$ are in $V'$ then the edge $uw$ was deleted from $G$ to form $G'$ so $\alpha(uw) = \lang \sk \rang$.  Thus $\bb^{i}_u - \bb^{i}_w \in \alpha(uw)$ so $\bb^{i}$ is a spline.
\item[(b)]If $\bb^{i}$ was produced by Case (b), then for every edge $uw \in E$ the difference $\bb^{i}_u - \bb^{i}_w$ is either zero or a nonzero element of $\lang \si \rang \cap \lang \sj \rang$.  Thus $\bb^{i}$ satisfies the GKM condition at each edge and hence is a spline.
\end{enumerate}

Now we show that $\cB = \{\mathbf{1}, \bb^{2}, \ldots, \bb^{{|V|}}\}$ generates the set $R_{G,\alpha}$ of all splines. Let $\bp \in R_{G,\alpha}$ be an arbitrary spline.  We will confirm that $\bp$ is an $R$-linear combination of elements of $\cB$ by explicitly identifying coefficients $r_2,\ldots,r_{|V|} \in R$ such that
\begin{equation}\label{eqn:twolabel}
\bp = \bp_{v_1}\mathbf{1} + r_2 \bb^{2} + \cdots + r_{|V|} \bb^{{|V|}}.
\end{equation}
We induct on $i$ using the inductive hypothesis that we have a \textit{unique} sequence of coefficients $r_1, r_2, \ldots, r_{i-1} \in R$ so that $r_1 \mathbf{1}+r_2 \bb^{2} + \cdots + r_{i-1} \bb^{i-1}$ agrees with $\bp$ on the first $i-1$ vertices.  For the base case $i=2$, the assignment $r_1 = \bp_{v_1}$ is the \textit{only} coefficient for which 
\[r_1 \mathbf{1}_{v_1} = \bp_{v_1}\]
since $\mathbf{1}_{v_1}$ is the multiplicative identity in $R$. 

Next assume that we have a (unique) sequence of coefficients $r_1, r_2, \ldots, r_{i-1}$ satisfying the inductive hypothesis. Let
\[\bq = \bp - r_1 \mathbf{1} - r_2 \bb^2 - r_3 \bb^3 - \cdots - r_{i-1} \bb^{i-1}\]
and note that $\bq_{v_j}=0$ for all $j < i$.  We know $\bq$ is a spline since $R_{G,\alpha}$ is an $R$-module.  We will show that there is a unique coefficient $r_i \in R$ solving the equation
\[\bq_{v_i} = r_i \bb^i\]
after which the inductive claim for $\bp$ follows. By Proposition~\ref{prop:order}, there is at least one $v_k \in N_G(\{v_{i}\})$ with $k < i$.  Suppose without loss of generality  $\alpha(v_kv_i) = \lang \si \rang$.  Since $\bq_{v_k}=0$ by construction, either
\begin{itemize}
    \item[(i)] $\bq_{v_i} \in \lang \si \rang \cap \lang \sj \rang^c$ or 
    \item[(ii)] $\bq_{v_i} \in \lang \si \rang \cap \lang \sj \rang$.
\end{itemize}
By definition $\bb^i_{v_i}$ is one of $\si, \sj,$ or $\lcm{\si,\sj}$.  Thus in Case (ii) we know $\bb^i_{v_i}$ divides $\bq_{v_i}$.  For the same to hold in Case (i) we must show $\bb^i_{v_i} = \si$.  Note first that there is no path whose edges are all labeled $\lang \sj \rang$ from $v_i$ to any $v_j$ with $j<i$.  Indeed, if there were, the GKM conditions along this $\lang \sj \rang$-labeled path would imply 
\[\bq_{v_{i}}-\bq_{v_j} \in \lang \sj \rang\]
By construction $\bq_{v_j} = 0$ so we would have $\bq_{v_i} \in \lang \si \rang \cap \lang \sj \rang$.  This contradicts the hypothesis of Case (i).  Moreover Lemma~\ref{lemma: 2-label cut-sets} shows that $\bb^i$ was produced by Case (a) so $\bb^i_{v_i} = \si$.  

In both cases $\bb^i_{v_i}$ divides $\bp_{v_i}$.  Since $R$ is a UFD there is a unique coefficient $r_i$ solving 
\[\bp_{v_i} = r_i \bb^i_{v_i}\]
By induction, Equation \eqref{eqn:twolabel} holds so $\cB$ generates $R_{G,\alpha}$ as an $R$-module.  

Finally, we check that $\cB$ is an MGS for $R_{G,\alpha}$.  The set $\cB$ has exactly $|V|$ elements.  Lemma~\ref{lem:folk} guarantees that every MGS for $R_{G,\alpha}$ has at least $|V|$ elements so the claim follows.
\end{proof}

\begin{example}\label{eg:2labalg}
Consider the edge-labeled graph
\begin{tikzcd}
v_1 \ar[r,dash,"\lang\si\rang"] \ar[d,swap,dash,"\lang\sj\rang"] & v_2 \ar[d,dash,"\lang\sj\rang"] \\
v_3 \ar[r,swap,dash,"\lang\si\rang"] & v_4
\end{tikzcd}.
Note that we have chosen an ordering on the vertices as in Theorem~\ref{thm:2labalg} (or Proposition~\ref{prop:order}).  

To produce $\bb^{2}$, we look at all vertices that are connected to $v_2$ by paths labeled exclusively $\lang j \rang$.  This gives the set $C'=\{v_2, v_4\}$.  Thus we are in Case (a), so $\bb^{2}$ is zero on $\{v_1, v_3\}$ and $\si$ otherwise.

Similarly, to find $\bb^{3}$ we get the connected component $C'=\{v_3, v_4\}$ and are again in Case (a).  In this case, $\bb^{3}$ is zero on $\{v_1, v_2\}$ and $\sj$ otherwise.

However, when constructing $\bb^{4}$ we find that $C' = \{v_2,v_4\}$.  Thus $\bb^{4}$ is $\lcm{\si,\sj}$ on $v_4$ and zero otherwise.

The set $\cB = \{\mathbf{1}, \bb^{2}, \bb^{3}, \bb^{4}\}$ is an MGS for $R_{G,\alpha}$ by Theorem~\ref{thm:2labalg}.
\end{example}

\begin{remark}\label{rem:sing}
In the theory of classical splines, the previous example (that is, a four-cycle with two distinct edge-labels) is the unique edge-labeled graph that is dual to a certain triangulation called, in the classical spline literature, an interior cell (or pinwheel triangulation) containing a {\em singular vertex}.  See the left picture in Figure~\ref{figure: two interior cells} for such a triangulation with singular vertex $v$.  In Section~\ref{sec:apps}, more will be said about the interesting role that such singular vertices play in (bounds for) dimension formulas of the space of classical splines.
\end{remark}

\section{Polynomial splines on cycles}
\label{sec:cycles}

In Section~\ref{sec:algorithm}, we produced MGSs for arbitrary connected $2$-labeled graphs.  In this section, we treat an arbitrary number of edge-labels, but we restrict the types of graphs and ideals under consideration.

\subsection{Degree sequences for splines}
\label{sec:degseq}

Let $R = \Bbbk[x_1,\ldots,x_m]$ with $\Bbbk$ a field, and let $(G,\alpha)$ be an arbitrary edge-labeled graph.  Recall that throughout this paper, we assume that all ideals in the image of $\alpha$ are principal (see Remark~\ref{rem:principalisok}).  We now add the assumption that the ideals are generated by homogeneous elements and introduce an invariant of $(G,\alpha)$ called the ``degree sequence". (As described in the introduction and in Section~\ref{sec:applications}, homogeneity is a very natural condition in geometric and classical (analytic) applications.)

\begin{definition}\label{def:homogdegseq}
An MGS $\cB=\{\bb^1,\ldots,\bb^n\}$ is called {\em homogeneous} if, for each $1\leq i\leq n$, every nonzero entry of $\bb^i$ is a homogeneous polynomial of the same degree, which we denote as $\deg \bb^i$.
\end{definition}

\begin{definition}
Let $\cB$ be a homogeneous MGS. For each $j \in \mathbb{Z}_{\geq 0}$, let $d_j= \big\vert \{\bb\in \cB\mid \deg\bb=j\}\big\vert$. Then the {\em degree sequence} of $\cB$ is defined as $\overline{d}_{\cB} = (d_0,d_1, d_2, d_3, \ldots)$.
\end{definition}

\begin{remark}
The degree sequence only has a finite number of nonzero entries.  For instance, when edge-labels are principal ideals, no generator need have larger degree than that of the product of the edge-labels.  In particular, $d_m = 0$ if $m$ is greater than the sum of the degrees of the generators of the edge-labels.  (See also \cite[Corollary 5.2]{GTV16}.)
\end{remark}

For the remainder of the paper, we only consider principal ideals generated by degree-two elements of the form $(x+ay)^2$ with $0 \ne a \in \Bbbk$.  (See Section~\ref{sec:applications} for how this case appears in the study of classical splines.)  For convenience, we denote these edge-labels using a sans-serif letter; for example, we write $\sfa := (x+ay)^2$.

\begin{example}
Let $\Bbbk$ be a field, and let $R = \Bbbk[x,y]$. Consider the edge-labeled graph $(G,\alpha)$ given by
\[
\begin{tikzcd}
v_1 \ar[r,dash,"\lang\sfa\rang"] \ar[d,swap,dash,"\lang\sfb\rang"] & v_2 \ar[d,dash,"\lang\sfb\rang"] \\
v_3 \ar[r,swap,dash,"\lang\sfa\rang"] & v_4
\end{tikzcd}
\]
where $0 \ne a,b \in \Bbbk$.  This is a specialization of the edge-labeled graph in Example~\ref{eg:2labalg}.  Theorem~\ref{thm:2labalg} asserts that
\[
\cB = \{\mathbf{1},(0,\sfa,0,\sfa),(0,0,\sfb,\sfb),(0,0,0,\sfa\sfb)\}
\]
is a homogeneous MGS for $R_{G,\alpha}$.  The degree sequence of $\cB$ is thus $\overline{d}_\cB = (1,0,2,0,1)$.
\end{example}

We next prove that the degree sequence is an invariant of an edge-labeled graph.

\begin{proposition}\label{prop:degseqinvt}
Let $(G,\alpha)$ be a connected edge-labeled graph with edges labeled by principal polynomial ideals with homogeneous generators. Let $\cB$ and $\cB'$ be two homogeneous MGSs for $R_{G,\alpha}$ with degree sequences $\overline{d}_{\cB}$ and $\overline{d}_{\cB'}$, respectively.  Then $\overline{d}_{\cB} = \overline{d}_{\cB'}$.
\end{proposition}

\begin{proof}
Let $\cB=\{\bb^1,\ldots, \bb^n\}$ and $\cB'=\{{\bb^1}',\ldots, {\bb^n}'\}$ be two homogeneous MGSs for $R_{G,\alpha}$ with degree sequences $\overline{d}_{\cB}=(d_0,d_1,d_2, \ldots)$ and $\overline{d}_{\cB'}=(d_0',d_1',d_2',\ldots)$.  (If needed, add terminal zeros so both sequences have the same length.) 
We show that $d_r=d_r'$ for each $r$ by induction on the index $r$ and prove as our base case that $d_0=d_0'$.  The degree-zero splines in $R_{G,\alpha}$ generate a $\Bbbk$-vector space.  The degree-zero splines in $\cB$ form an MGS for the degree-zero splines in $R_{G,\alpha}$, and likewise for $\cB'$.  Since MGSs in vector spaces are bases, and in particular have the same number of elements, the base case of our induction holds.  Now assume that $d_0 = d_0',d_1=d_1', \ldots, d_{r-1} = d_{r-1}'$.

Given ${\bb^i}'\in\cB'$, we can write ${\bb^i}'= k_{i,1}\bb^1+k_{i,2}\bb^2+\cdots+k_{i,n}\bb^n$ for some coefficients $k_{i,1},\ldots,k_{i,n}\in R$.  Note that
\[
\sum_{\substack{\bb^j \in \cB\\\deg {\bb^i}' < \deg \bb^j}} k_{i,j}\bb^j = 0,
\]
so we may assume $k_{i,j}=0$ for all $\deg {\bb^i}' < \deg \bb^j$.
It follows that \[\{{\bb^i}'\in\cB' \mid \deg {\bb^i}' \leq r\}\subseteq \text{span}(\{\bb^i\in\cB \mid \deg \bb^i \leq r\}).\]  
A symmetric argument shows that
\[\{\bb^i\in\cB \mid \deg \bb^i \leq r\}\subseteq \text{span}(\{{\bb^i}'\in\cB' \mid \deg {\bb^i}' \leq r\}).\] 
This contradicts minimality of $\cB$ or $\cB'$ unless $\sum_{i=0}^r d_i = \sum_{i=0}^{r} d_i'$.  By the inductive hypothesis, this implies $d_r = d_r'$.
\end{proof}

\subsection{Linear algebraic background} \label{sec:linalg}

We use the following two linear-algebraic results about the $\Bbbk$-vector space of polynomials.

\begin{lemma}\label{lem:linearindependence}
Let $a,b,c,d,D\in \Bbbk$, with $a, b, c$ distinct. Then we can find unique $A,B,C\in \Bbbk$ such that
\begin{equation}\label{eq:linalg1}
A\sfa + B\sfb + C\sfc = D\sfd.
\end{equation}
\end{lemma}

\begin{proof}
We rewrite Equation~\eqref{eq:linalg1}, collecting coefficients, as 
\[
(A+B+C)x^2+(2aA+2bB+2cC)xy+(a^2A+b^2C+c^2C)y^2=Dx^2+2Ddxy+Dd^2y^2.
\]

Solving for $A,B,C$ amounts to solving the following system of linear equations:
\begin{align*}
A+B+C&=D\\
2aA+2bB+2cC&=2dD\\
a^2A+b^2B+c^2C&=d^2D.
\end{align*}
The coefficient matrix
\[
\begin{bmatrix}
1 & 1 & 1\\
2a & 2b & 2c\\
a^2 & b^2 & c^2
\end{bmatrix}
\]
can be reduced to the identity matrix via elementary row operations. (Some of the row operations require division, but we avoid division by zero because $a, b, c$ are distinct.) This implies that the coefficient matrix is invertible; thus, there exists a unique solution to the system of equations.
\end{proof}

\begin{lemma}\label{lem:degree3polybasis}
Let $a,b,c,C_1,C_2\in \Bbbk$, with $a$ and $b$ distinct.  
There exist unique $A_1,A_2,B_1,B_2\in \Bbbk$ such that
\begin{equation}\label{eq:linalg2}
(A_1x + A_2y)\sfa + (B_1x + B_2y)\sfb = (C_1x + C_2y)\sfc.
\end{equation}
\end{lemma}
\begin{proof}
Expanding Equation~\eqref{eq:linalg2} and equating coefficients of like terms leads to the following system of linear equations:
\begin{align*}
A_1+B_1&=C_1\\
2aA_1+A_2+2bB_1+B_2&=2cC_1+C_2\\
a^2A_1+2aA_2+b^2B_1+2bB_2&=c^2C_1+2cC_2\\
a^2A_2+b^2B_2&=c^2C_2.
\end{align*}
The coefficient matrix
\[
\begin{bmatrix}
1 & 0 & 1 & 0\\
2a & 1 & 2b & 1\\
a^2 & 2a & b^2 & 2b\\
0 & a^2 & 0 & b^2
\end{bmatrix}
\]
can be reduced to the identity matrix via elementary row operations. (Some of the row operations require division, but we avoid division by zero because $a\neq b$.) This implies that the coefficient matrix is invertible; thus, there exists a unique solution to the system of equations.
\end{proof}

\subsection{Constructions to reduce graphs}
\label{sec:reductionlemmas}

A product module has a collection of forgetful maps to different factors in the module.  Suppose $(G', \alpha')$ is an edge-labeled graph obtained from another edge-labeled graph $(G, \alpha)$ by adding a single vertex $v$ together with some labeled edges from $v$ to vertices in $G$.  Then we can use the forgetful map to relate the splines on $(G',\alpha')$ to those on $(G,\alpha)$.  

This is what we do in the next result.  We then specialize to the case of cycles in Corollary~\ref{cor:repeatedge}.  We note that Lemma~\ref{lem:splinereducing} below also applies to general edge-labelings $\alpha$, which we will restrict in more ways throughout this section.

\begin{lemma}\label{lem:splinereducing}
Suppose that $(G,\alpha)$ and $(G',\alpha')$ are  edge-labeled graphs with vertices $V'=V \cup \{v\}$, edges 
\[E'=E \cup \{vu \mid u\in U\}\] 
where $U\subseteq V$ is nonempty, and edge-labeling $\alpha'|_E = \alpha$.  Then the projection map $\bigoplus_{u\in V'} R\rightarrow \bigoplus_{u\in V} R$ induces an $R$-module homomorphism $\varphi\co R_{G',\alpha'} \ra R_{G,\alpha}$, and
\[R_{G',\alpha'} \cong \ker \varphi \oplus \im \varphi.\]

Moreover, suppose that every pair $u,u'\in U$ is connected by a path of edges in $E$ all labeled $I$, and suppose that $\alpha'(vu)=I$ for all $u\in U$. Then
\[R_{G',\alpha'} \cong I \oplus R_{G,\alpha}.\]
\end{lemma}

\begin{proof}
We first show that restricting a spline $\bp$ in $R_{G',\alpha'}$ to the set of vertices $V$ produces a spline in $R_{G,\alpha}$.  Indeed, for each edge $uu' \in E$ we have
\[\bp_u - \bp_{u'} \in \alpha'(uu') = \alpha(uu').\]
Thus the projection map induces an $R$-module homomorphism $\varphi\co R_{G', \alpha'} \ra R_{G,\alpha}$.
We conclude that 
\[R_{G',\alpha'} \cong \ker \varphi \oplus \im \varphi.\]

Now we consider the special case where every pair $u,u'\in U$ is connected by a path of edges in $E$ all labeled $I$, and $\alpha'(vu)=I$ for all $u\in U$. Note that $\ker\varphi$ consists of all splines in $R_{G',\alpha'}$ that are zero at all of $V$.  Consider a spline in $\ker \varphi$.  If at least one edge incident to $v$ is labeled $I$, then the vertex $v$ must be labeled by an element of $I$ by the GKM condition; if all edges incident to $v$ are labeled $I$, then every element of $I$ works. Thus $\ker\varphi\cong I$. 

Given a spline $\bq \in R_{G,\alpha}$, we define $\bp\in R_{G',\alpha'}$ such that $\varphi(\bp) = \bq$ according to the rule $\bp_{u} = \bq_{u}$ for all $u \in V$ and $\bp_{v} = \bq_{u'}$ for some $u' \in U$. The GKM condition implies that $\bq_{u}-\bq_{u'}\in I$ for any $u\in  U$, since $u$ and $u'$ are connected by a path of edges labeled $I$ by hypothesis. Thus, we have $\bp_{u}-\bp_{v}=\bq_u-\bq_{u'}\in I$ for all $u\in U$. 
By inspection of the GKM conditions, we conclude $\bp \in R_{G',\alpha'}$ and thus $\im\varphi\cong R_{G,\alpha}$.
\end{proof}

We can (and will) use Lemma \ref{lem:splinereducing} to eliminate those vertices whose incident edges all have the same label.  This leads us to the following definition.

\begin{definition}\label{def:reduced}
An edge-labeled graph is called {\em reduced} if no two edges that are incident to the same vertex have the same edge-label.
\end{definition}

We note that the edge-labels in a reduced cycle have to be at least somewhat evenly distributed, in the following sense.

\begin{lemma}\label{lem:seqof3}
Suppose $(G,\alpha)$ is a reduced edge-labeled graph.  Moreover, suppose $(G,\alpha)$ contains $n$ distinct vertices $v_0, v_1, v_2, \ldots, v_n=v_0$ that form a cycle and assume that $|\alpha(E)| \geq 3$. Then there is at least one sequence of three successive distinct edge-labels on the edges $v_1v_2, v_2v_3,\ldots, v_{n-1}v_n,v_nv_1$. 
\end{lemma}

\begin{proof}
Read clockwise around the cycle starting at an arbitrary edge, and suppose the first two edges are labeled $I$ and $J$.  If a sequence of three successive edge-labels does not contain three distinct edge-labels, then it must alternate between two of them since the graph is reduced.  An edge-label that is neither $I$ nor $J$ appears somewhere on the graph by hypothesis of at least three distinct edge-labels.  Look at the first occurrence of this edge-label in the sequence; the two edges preceding it have labels from the set $\{I, J\}$, without repetition.  This proves the claim.
\end{proof}

In the next lemma, we refine Lemma~\ref{lem:splinereducing} to keep track of MGSs. While we use this lemma to prove results about polynomial edge-labelings, the same proof works for edge-labelings over UFDs as long as each edge is labeled by a principal ideal.

\begin{lemma}\label{lem:repeatedgeMGS}
Let $(G,\alpha)$ and $(G',\alpha')$ be defined as in Lemma~\ref{lem:splinereducing}, with the condition that every pair $u,u'\in U$ is connected by a path of edges in $E$ all labeled $\lang\si\rang$ and $\alpha'(vu)=\lang\si\rang$ for all $u\in U$. Let $|V|=n$ and fix some $u'\in U$. If $\cB = \{\bb^{1}, \bb^{2}, \ldots, \bb^{n}\}$ is an MGS for $R_{G,\alpha}$, then $\cB' = \{{\bb^{1}}',{\bb^{2}}',\ldots,{\bb^{n}}',\bb^{n+1}\}$,  where
\[
{\bb^{i}_{u}}' = \left\{\begin{array}{ll}\bb^{i}_{u} & \text{if } u\in V,\\\bb^{i}_{u'} & \text{if } u = v,\end{array}\right.
\]
and
\[
\bb^{n+1}_{u} = \left\{\begin{array}{ll}0 & \text{if } u\in V,\\ \si & \text{if } u=v,\end{array}\right.
\]
is an MGS for $R_{G',\alpha'}$.
\end{lemma}

\begin{proof}
It is clear that $\bb^{n+1}$ is a spline in $R_{G',\alpha'}$.  If $\varphi$ is the map from Lemma~\ref{lem:splinereducing}, then $\varphi({\bb^i}')=\bb^i$ for all $1\leq i\leq n$ by construction, so (the last paragraph of the proof of) Lemma~\ref{lem:splinereducing} implies that ${\bb^i}'\in R_{G',\alpha'}$ again for all $1 \leq i \leq n$.  

We now show that $\cB'$ is a generating set for $R_{G',\alpha'}$.  Each spline $\bp \in R_{G',\alpha'}$ satisfies 
\begin{equation}\label{eqn:lastcoordinate}
\bp_{v} = \bp_{u'} + k\si
\end{equation}
for some $k \in R$ by the GKM condition for the edge $vu'$. The spline $\varphi(\bp)$ is an element of $R_{G,\alpha}$ with
\begin{equation}\label{eqn:samespline}
\varphi(\bp)_{u} = \bp_{u} \qquad \text{for all } u\in V.
\end{equation}
Because $\cB$ is an MGS for $R_{G,\alpha}$, we can write $\varphi(\bp)$ as a linear combination
\begin{equation}\label{eqn:lincombop}
\varphi(\bp) = r_1\bb^{1} + r_2 \bb^{2} + \cdots + r_n \bb^{n},
\end{equation}
where each $r_i \in R$.  For all $u\in V$, we have
\begin{align*}
\bp_{u} &= \varphi(\bp)_{u} & \text{by Equation } \eqref{eqn:samespline},\\
&= r_1\bb^{1}_{u} + r_2\bb^{2}_{u} + \cdots + r_n\bb^{n}_{u} + k\cdot 0 & \text{by Equation } \eqref{eqn:lincombop},\\
&= r_1{\bb^{1}_{u}}' + r_2{\bb^{2}_{u}}' + \cdots + r_n{\bb^{n}_{u}}' + k\bb^{{n+1}}_{u} & \text{by the definition of } \cB'.
\end{align*}
Furthermore,
\begin{align*}
\bp_{v} &= \bp_{u'} + k\si & \text{by Equation } \eqref{eqn:lastcoordinate},\\
&= r_1\bb^{1}_{u'} + r_2\bb^{2}_{u'} + \cdots + r_n\bb^{n}_{u'} + k\si & \text{by the previous argument},\\
&= r_1{\bb^{1}_{v}}' + r_2{\bb^{2}_{v}}' + \cdots + r_n{\bb^{n}_{v}}' + k\bb^{{n+1}}_{v} & \text{by the definition of } \cB'.
\end{align*}
We have obtained the equation
\[
\bp = r_1{\bb^{1}}' + r_2{\bb^{2}}' + \cdots + r_n{\bb^{n}}' + k\bb^{n+1},
\]
which assures that $\cB'$ is a generating set for $R_{G',\alpha'}$.

Moreover, the set $\cB'$ is an MGS by Lemma~\ref{lem:folk} because it consists of $n+1$ elements and  $G'$ is a graph with $n+1$ vertices.
\end{proof}

We now apply the ideas in the previous lemma to the case of cycles, which is the special case on which we focus.

\begin{corollary}\label{cor:repeatedge}
Let $(C_n,\alpha_n)$ be an edge-labeled $n$-cycle. Create an edge-labeled $(n+1)$-cycle $(C_{n+1},\alpha_{n+1})$ from $(C_n,\alpha_n)$ by inserting a vertex $v_{n+1}$ into the edge $v_n v_1$ with both new edges $v_nv_{n+1}$ and $v_{n+1}v_1$ labeled the same as $v_nv_1$ was.  Then 
\[R_{C_{n+1}, \alpha_{n+1}} \cong \alpha_{n}(v_nv_1) \oplus R_{C_n, \alpha_n}.\]
Moreover, suppose $(C_n,\alpha_n)$ has edges labeled with principal ideals generated by homogeneous polynomials, that $(C_n,\alpha_n)$ has MGS $\cB$,
and that the generator of the edge-label $\alpha_n(v_nv_1)$ is a homogeneous polynomial of degree $e$.  Then
$(C_{n+1},\alpha_{n+1})$ has an MGS $\cB'$ that 
\begin{enumerate}
    \item extends $\cB$ in the sense that if $\varphi$ is the map from Lemma~\ref{lem:splinereducing} then $\varphi(\cB') \supseteq \cB$,
    \item has exactly one more generator than $\cB$ and the degree of this additional generator is $e$, and
    \item is minimal in the sense that if $\cB''$ is any other generating set that extends $\cB$ then $\cB''$ has at least one more element of degree $e$ than $\cB$ (and possibly other additional elements of other degrees).
\end{enumerate}
In particular, the degree sequence of $\cB'$ satisfies
\[\overline{d}_{\cB'} = \overline{d}_{\cB}+\left(0^{e-1},1,0^{m-e}\right). \]
\end{corollary}

\begin{proof}
Taking $(G,\alpha)$ to be $(C_n,\alpha_{n})$ and $(G',\alpha')$ to be the (non-cyclic) edge-labeled graph formed from $(C_n,\alpha_n)$ by adding a new vertex $v_{n+1}$ and new edges $v_nv_{n+1}$ and $v_{n+1}v_1$ labeled the same as $v_nv_1$, we can apply Lemma~\ref{lem:splinereducing} to conclude
\[R_{G',\alpha'} \cong \alpha_{n}(v_nv_1) \oplus R_{C_n, \alpha_n}.\]
Note that $R_{G',\alpha'} \subseteq R_{C_{n+1},\alpha_{n+1}}$ because $(G',\alpha')$ consists of the edge-labeled graph $(C_{n+1},\alpha_{n+1})$ together with precisely one additional edge.  The three edges $v_nv_1$, $v_nv_{n+1}$, and $v_{n+1}v_1$ in $(G',\alpha')$ all have the same label, so every spline in $R_{C_{n+1},\alpha_{n+1}}$ satisfies the GKM conditions on $(G',\alpha')$. Thus $R_{G',\alpha'} \cong R_{C_{n+1},\alpha_{n+1}}$.  In particular, if $\alpha_{n}(v_nv_1)$ is a principal ideal generated by a homogeneous polynomial of degree $e$ and if $\cB$ is an MGS of $(C_n, \alpha_n)$, then  Lemma~\ref{lem:repeatedgeMGS} constructs an MGS for $(C_{n+1}, \alpha_{n+1})$ that satisfies Conditions (1) and (2) of our claim.  The explicit description of the degree sequence $\overline{d}_{\cB'}$ of $(C_{n+1},\alpha_{n+1})$ follows from the definition of degree sequence and from Conditions (1) and (2).  

The minimality in Condition (3) results from the direct sum decomposition 
\[R_{C_{n+1}, \alpha_{n+1}} \cong \alpha_{n}(v_nv_1) \oplus R_{C_n, \alpha_n}\]
as follows.  Suppose $\cB'$ generates $R_{C_{n+1}, \alpha_{n+1}}$ and extends $\cB$.  Lemma~\ref{lem:splinereducing} identifies each spline in the image of $\varphi$ with an element of the subring $0 \oplus R_{C_{n,\alpha_{n}}}$ and each spline in $\ker \varphi$ with an element of $\alpha_n(v_nv_1) \oplus 0$.  The set $\cB \subseteq \varphi(\cB')$ is contained in $0 \oplus R_{C_{n,\alpha_{n}}}$ so $\cB'$ contains at least one element $\bb \in \cB' \cap \ker \varphi$ to generate the first summand $\alpha_n(v_nv_1) \oplus 0$.  This shows $|\cB'| > |\cB|$.

The spline $\bb$ satisfies $\bb_{v_i}=0$ for all $i \neq n+1$ by definition of $\ker \varphi$ and is identified with $(\bb_{v_{n+1}},0) \in \alpha_n(v_n,v_1) \oplus 0$ in the direct sum decomposition. By hypothesis, the minimal generator $p \in \alpha_n(v_nv_1)$ is a homogeneous polynomial of degree $e$.  Since $\bb_{v_{n+1}}$ is divisible by $p$ we conclude $\bb$ has degree at least $e$ as desired.
      \end{proof}

\subsection{Producing an MGS for polynomial edge-labeled cycles}
\label{sec:algpoly}

We now construct an algorithm that produces a homogeneous MGS for cycles whose edges are labeled by principal polynomial ideals with generator of the form $\sfa := (x+ay)^2$ for $a \neq 0$.  Part of our proof proceeds by induction; the following lemma proves the base case of a triangle.

\begin{lemma}\label{lem:trianglecase}
Let $(G,\alpha)$ be a $3$-cycle with edge-labeling $\alpha\co E\ra \cI$ having $\alpha(v_1v_2)=\lang \sfa\rang$, $\alpha(v_2v_3)=\lang \sfb\rang$, and $\alpha(v_3v_1)=\lang \sfc\rang$ so that $a, b, c$ are all distinct. Let $f_1,f_2,g_1,g_2\in \Bbbk[x,y]$ denote the homogeneous degree-one polynomials with $x\sfa=f_1\sfb+g_1\sfc$ and $y\sfa=f_2\sfb+g_2\sfc$, whose existence is guaranteed by Lemma \ref{lem:degree3polybasis}.  Then the set
\[
\cB=\{\bid,\bb^{2},\bb^{3}\}=\{\bid, (0,x\sfa,g_1\sfc),(0,y\sfa,g_2\sfc)\}
\] 
is a homogeneous basis for $R_{G,\alpha}$.
\end{lemma}

\begin{proof}
We will prove that $\cB$ is a homogeneous MGS that is also free.  Note that $(0,x\sfa,g_1\sfc)$ is a spline: the GKM condition on the edges labeled $\sfa$ and $\sfc$ are trivially satisfied, and the condition on the edge labeled $\sfb$ is satisfied because $x\sfa=f_1\sfb+g_1\sfc$.   The same argument shows that $(0,y\sfa,g_2\sfc)$ is a spline.  Moreover, the GKM condition for spline $\bb^{2}$ on edge $v_2v_3$ implies $g_1 \neq x$ since  $\sfb = (x+by)^2$ cannot divide the polynomial
\[x\sfa - x\sfc = xy(2ax-2cx+a^2y-c^2y)\]  
by direct computation (or by noting that both $\sfb$ and the right-hand side of the displayed equation are factored into irreducibles, and that the polynomial ring $\Bbbk[x,y]$ is a UFD). A similar argument shows $g_2 \neq y$.

Now we demonstrate that $\cB$ generates the spline $(0,0,\sfb\sfc)$. We have 
\[
yg_1\sfc=xy\sfa-yf_1\sfb \quad \text{and} \quad xg_2\sfc=xy\sfa-xf_2\sfb.
\] 
Subtracting, we obtain the equality
\[(yg_1-xg_2)\sfc=(xf_2-yf_1)\sfb.\] 
If $(yg_1-xg_2)\sfc=0$ then the degree-two factor $yg_1-xg_2$ is identically zero, and so $g_1=rx$ and $g_2=ry$ for some scalar $r$. Plugging this back into the equation $yg_1\sfc=xy\sfa-yf_1\sfb$ and then rearranging, we obtain $x(\sfa-r\sfc)=f_1\sfb$ and similarly $y(\sfa-r\sfc)=f_2\sfb$. Multiplying these two equations by $y$ and $x$ respectively, we obtain $yf_1=xf_2$.  Analyzing degree constraints once more, we conclude $f_1=sx$ and $f_2=sy$ for some scalar $s$. Plugging this back into the equation $x \sfa=f_1\sfb+g_1\sfc$, we have $x\sfa=sx\sfb+rx\sfc$.  In particular $\sfa=s\sfb+r\sfc$, which contradicts the linear independence of $\sfa$, $\sfb$, and $\sfc$ over $\Bbbk$ proved in Lemma~\ref{lem:linearindependence}.

Thus $(yg_1-xg_2)\sfc$ is a homogeneous degree-four polynomial that is divisible by both $\sfb$ and $\sfc$.  It must be a scalar multiple of $\sfb\sfc$ because $\sfb$ and $\sfc$ have no irreducible factors in common. Consequently, the spline 
\[
\bq:=y\bb^{2}-x\bb^{3} = (0,0,(yg_1-xg_2)\sfc)
\]
is a nonzero scalar multiple of $(0,0,\sfb\sfc)$.

Now we show the generators are actually free, namely that if 
\[ p_1 \bid + p_2 \bb^{2} - p_3 \bb^{3} = (0,0,0) \]
then $p_i = 0$ for all $i \in \{1,2,3\}$.  The first coordinate shows that $p_1 = 0$ since on the left-hand side we have
\[ \left( p_1 \bid + p_2 \bb^2 - p_3 \bb^3 \right)_{v_1} = p_1 \bid_{v_1} + 0 - 0 = p_1.\]
Using $p_1 = 0$ and the explicit equations for $\bb^1, \bb^2$, we obtain
\[ (0,p_2x \sfa - p_3y \sfa, p_2g_1 \sfc - p_3 g_2 \sfc) = (0,0,0).\]
Since $p_2x \sfa - p_3 y \sfa = 0$ in a UFD, we conclude as above that $x$ divides $p_3$ and $y$ divides $p_2$.  Write $p_3 = p_3'x$ and $p_2=p_2'y$.  Then we have 
\[p_2' yx \sfa - p_3' xy \sfa = (p_2' - p_3') xy \sfa = 0\] 
and so $p_2' = p_3'$.  Now examining the last coordinate, we see
\[ p_2'yg_1 \sfc - p_3' x g_2 \sfc = \left(p_2'\right)\left( (yg_1 - x g_2) \sfc \right) =0.\]
We just proved that $(yg_1 - xg_2) \sfc$ is a nonzero scalar multiple of $\sfb \sfc$, so this entry is zero if and only if $p_2' = 0$.  Hence all $p_i$ are zero, as desired.

Finally we show that $\cB$ generates an arbitrary spline $\bp \in R_{G,\alpha}$. We have 
\[\bp-\bp_{v_1}\bid=(0,\bp_{v_2}-\bp_{v_1},\bp_{v_3}-\bp_{v_1}).\] 
By the GKM conditions on edges $v_1v_2$ and $v_3v_1$, we have $\bp_{v_2}-\bp_{v_1}=k\sfa$ and $\bp_{v_3}-\bp_{v_1}=\ell \sfc$ for some $k,\ell \in \Bbbk[x,y]$. The GKM condition on edge $v_2v_3$ gives the equation
\[(\bp_{v_2} - \bp_{v_1}) - (\bp_{v_3} - \bp_{v_1}) = k\sfa - \ell \sfc = \ell' \sfb\]
for some $\ell' \in \Bbbk[x,y]$.  Lemma \ref{lem:linearindependence} showed that $\sfa$, $\sfb$, and $\sfc$ are linearly independent over the base field $\Bbbk$, so the only scalar solution to $k\sfa-\ell \sfc - \ell' \sfb=0$ is $k=\ell=\ell'=0$. Thus $k,\ell,\ell'$ are polynomials without constant terms.  
Assume that $h_1,h_2 \in \Bbbk[x,y]$ satisfy $k=h_1x+h_2y$. We have \[\bp_{v_2}-\bp_{v_1}=(h_1x+h_2y)\sfa\]
and 
\[\bp-\bp_{v_1}\bid-h_1\bb^{2}-h_2\bb^{3}=(0,0,\bp_{v_3}-\bp_{v_1}-h_1g_1\sfc-h_2g_2\sfc).\] 
The nonzero entry in this spline must be a multiple of both $\sfb$ and $\sfc$ by the GKM conditions on edges $v_2v_3$ and $v_3v_1$, respectively. Hence \[\bp-\bp_{v_1}\bid-h_1\bb^{2}-h_2\bb^{3}=t\bq\] 
for some $t\in \Bbbk[x,y]$ because we showed above that $\bq$ is a scalar multiple of $(0,0,\sfb \sfc)$. We conclude that $\cB$ generates $R_{G,\alpha}$. Lemma~\ref{lem:folk} asserts that $\cB$ is an MGS as desired.
\end{proof}

The heart of the proof of Theorem~\ref{thm:polysplines}, our main theorem about cycles, is the following lemma.  After proving the lemma, Theorem~\ref{thm:polysplines} will follow easily by applying the reduction lemmas from Section~\ref{sec:reductionlemmas}.

\begin{lemma}\label{lem:reduced}
Let $(G,\alpha)$ be an edge-labeled $n$-cycle containing a sequence of three successive distinct edge-labels. Order the vertices $v_0, v_1, v_2, \ldots, v_n = v_0$ of $(G,\alpha)$ clockwise around the cycle such that $\alpha(v_{i-1}v_{i}) = \langle \sfa_i \rangle$ and $a_{n-1}, a_{n},$ and $a_{1}$ are all distinct. 
 
We give an explicit homogeneous MGS $\cB = \{\bid, \bb^{2}, \ldots, \bb^{n}\}$ for $R_{G,\alpha}$ as follows.
For every $1 < i \le n-2$, let $a_{i,n-1}, a_{i,n}, a_{i,1} \in \Bbbk$ be the base field elements with $\sfa_i = a_{i,n-1} \sfa_{n-1} + a_{i,n} \sfa_n + a_{i,1} \sfa_1$, whose existence is guaranteed by Lemma~\ref{lem:linearindependence}, and define $\bb^{i}$ by
\[
\bb^{i}_{v_j} = \left\{\begin{array}{ll}0 & \text{if } j < i,\\
\sfa_{i} & \text{if } i \le j \le n-2,\\
\sfa_{i} - a_{i,n-1}\sfa_{n-1} & \text{if } j = n-1,\\
\sfa_{i} - a_{i,n-1}\sfa_{n-1} - a_{i,n}\sfa_n & \text{if } j = n.\end{array}\right.
\]
As in Lemma~\ref{lem:trianglecase}, let $f_1,f_2,g_1,g_2\in \Bbbk[x,y]$ denote the homogeneous degree-one polynomials with $x\sfa_{n-1}=f_1\sfa_n+g_1\sfa_{1}$ and $y\sfa_{n-1}=f_2\sfa_n+g_2\sfa_1$ that are guaranteed by Lemma \ref{lem:degree3polybasis}.  Define $\bb^{n-1}$ by
\[
\bb^{n-1}_{v_j} = \left\{\begin{array}{ll}0 & \text{if } j \le n-2,\\
x\sfa_{n-1} & \text{if } j = n-1,\\
g_1\sfa_1 & \text{if } j = n,\\ \end{array}\right.
\]
and $\bb^{n}$ by
\[
\bb^{n}_{v_j} = \left\{\begin{array}{ll}0 & \text{if } j \le n-2,\\
y\sfa_{n-1} & \text{if } j = n-1,\\
g_2\sfa_1 & \text{if } j = n.\\ \end{array}\right.
\]
Then $\cB=\{\bid,\bb^{2},\ldots,\bb^{n}\}$ is a homogeneous basis for $R_{G,\alpha}$ as a free module over the polynomial ring.  Consequently, the degree sequence of $(G,\alpha)$ is $(1,0,n-3,2)$.
\end{lemma}

\begin{proof}
We will check that $\cB$ is a homogeneous MGS and that it is free, whence we will conclude that it is a homogeneous basis for the free module $R_{G,\alpha}$.

We first check that $\bb^{i}$ is a spline in $R_{G,\alpha}$ for all $i > 1$. For the $\bb^{i}$ with $1 < i \le n-2$ this is clear by the definition of $\bb^{i}$, and for $\bb^{n-1}$ (respectively $\bb^n$) this follows from the GKM condition together with rewriting the defining equation as $x\sfa_{n-1} - g_1 \sfa_{1}=f_1\sfa_n$ (respectively as $y\sfa_{n-1} - g_2 \sfa_1 = f_2 \sfa_n$).

Now we show that $\cB$ generates an arbitrary spline $\bp \in R_{G,\alpha}$.  We claim that there exist $r_1, r_2,\ldots,r_n \in \Bbbk[x,y]$ such that
\begin{equation}\label{eqn:inductionproof}
\bp = r_1\bid + r_2\bb^{2} + \cdots + r_n\bb^{n}.
\end{equation}
For this, it is sufficient to prove that for all $1 \le m \le n$, we can find coefficients $r_1, r_2, \ldots, r_m \in \Bbbk[x,y]$ such that the spline $r_1 \mathbf{1} + r_2 \bb^2 + \cdots + r_m \bb^m$ agrees with $\bp$ when evaluated at the first $m$ vertices.
We will use induction up to $n-2$, then deal with $\bb^{{n-1}}$ and $\bb^{{n}}$ separately.  For the base case, note that the spline $\bp - \bp_{v_1}\bid$ has $(\bp - \bp_{v_1}\bid)_{v_1} = 0$.  The inductive hypothesis asserts that we can find coefficients $r_1, r_2, \ldots, r_m \in \Bbbk[x,y]$ with $m < n-2$ so that the spline $r_1\mathbf{1}+r_2\bb^{2}+\cdots+r_m\bb^{m}$ agrees with $\bp$ when evaluated at the first $m$ vertices.
In other words, assume we have found $r_1, r_2,\ldots,r_m \in \Bbbk[x,y]$ with $m < n-2$ such that
\[
(\bp - r_1\bid - r_2 \bb^{2} - \cdots - r_m \bb^{m})_{v_j} = 0
\]
for all $j \le m$.  Thus by the GKM condition on edge $v_mv_{m+1}$, there exists $r_{m+1} \in \Bbbk[x,y]$ such that
\[
(\bp - r_1\bid - r_2 \bb^{2} - \cdots - r_m \bb^{m})_{v_{m+1}} = r_{m+1}\sfa_{m+1}.
\]
Hence the spline $r_1\mathbf{1}+r_2\bb^{2}+\cdots+r_m\bb^{m}+r_{m+1}\bb^{{m+1}}$ agrees with $\bp$ when evaluated at the first $m+1$ vertices, as desired.  By induction, we have produced $r_1,r_2,\ldots,r_{n-2} \in \Bbbk[x,y]$ such that  $r_1\mathbf{1}+r_2\bb^{2}+\cdots+r_{n-2}\bb^{{n-2}}$ agrees with $\bp$ when evaluated at the first $n-2$ vertices.

To conclude the proof, we essentially use the same argument as in the proof of Lemma~\ref{lem:trianglecase}.  Indeed, suppose $(T,\alpha')$ is the edge-labeled $3$-cycle with vertices $v_1, v_{n-1}, v_n$, and with edge-labeling given by $\alpha'(v_1v_{n-1}) = \lang\sfa_{n-1}\rang$, $\alpha'(v_{n-1}v_{n}) = \lang \sfa_n \rang$, and $\alpha'(v_nv_1) = \lang\sfa_{1}\rang$.  Let $\mathcal{G}$ be the subset of $R_{G,\alpha}$ in which all vertices $v_1,v_2,\ldots,v_{n-2}$ are labeled zero.  Note that $\mathcal{G}$ is isomorphic to the subset of $R_{T,\alpha'}$ in which vertex $v_1$ is labeled zero, via the map $\mathcal{G} \rightarrow R_{T,\alpha'}$ that erases the initial $n-1$ zeros from each spline $\bp \in \mathcal{G}$.  Thus inserting $n-1$ leading zeros into the nontrivial generators from Lemma~\ref{lem:trianglecase} gives generators for $\mathcal{G}$.  

It follows that $\cB$ generates $R_{G,\alpha}$. Indeed, we first proved that for any spline $\bp \in R_{G,\alpha}$ we can find a {\em unique} linear combination of the splines $\{\bid, \bb^1, \ldots, \bb^{n-2}\}$ so that $\bp - r_0 \bid -  \sum_{i=1}^{n-2} r_i \bb^i$ is zero when evaluated at the first $n-2$ vertices.  Lemma \ref{lem:trianglecase} then proved that if a spline in $R_{G,\alpha}$ is zero at the first $n-2$ vertices, it is {\em uniquely} generated by $\{\bb^{n-1}, \bb^n\}$. 
The generating set $\cB$ is thus minimal and a free set of generators for the module of splines $R_{G,\alpha}$ over the polynomial ring. 

Finally, the statement on the degree sequence follows because $\bid$ is a degree-zero spline, $\bb^{i}$ is a degree-two spline for all $1 < i \le n-2$, and $\bb^{{n-1}}$ and $\bb^{n}$ are both degree-three splines.
\end{proof}

\begin{theorem}\label{thm:polysplines}
Let $(C_n,\alpha_n)$ be an $n$-cycle with three or more distinct (not necessarily successive) edge-labels.  The following algorithm constructs a homogeneous MGS $\cB_n$ for $R_{C_n,\alpha_n}$:
\begin{enumerate}
    \item Let $C_{n-k}$ be the reduced cycle with edge-labeling $\alpha_{n-k}$ obtained from $C_n$ by eliminating vertices whose incident edges have the same label.
    \item Let $\cB_{n-k}$ be the homogeneous MGS for $R_{C_{n-k},\alpha_{n-k}}$ from Lemma~\ref{lem:reduced}.
    \item Create $\cB_n$ from $\cB_{n-k}$ by successively reinserting vertices on repeated edges according to Corollary~\ref{cor:repeatedge}.
\end{enumerate}
\end{theorem}

\begin{proof}
Suppose $C_n = (V_n, E_n)$ is a cycle in which $k$ vertices are incident to two edges with the same label.  Without loss of generality, label the vertices sequentially around the cycle so that $v_n$ is one of the vertices incident to two edges with the same label. Using Corollary~\ref{cor:repeatedge}, write
\[R_{C_n,\alpha_n} \cong \alpha_{n-1}(v_{n-1}v_1) \oplus R_{C_{n-1},\alpha_{n-1}}\]
where $(C_{n-1}, \alpha_{n-1})$ is an edge-labeled $(n-1)$-cycle with only $k-1$ vertices are incident to two edges with the same label.  Repeat this process until no vertices are incident to edges with the same label, leaving a reduced $(n-k)$-cycle $(C_{n-k},\alpha_{n-k})$ with edge-labeling $\alpha_{n-k}$ obtained from $\alpha_n$.  

Lemma~\ref{lem:seqof3} proves that the reduced cycle $C_{n-k}$ contains three successive distinct edge-labels.  Thus we may apply Lemma~\ref{lem:reduced} to obtain the homogeneous MGS $\cB_{n-k}$ for $R_{C_{n-k},\alpha_{n-k}}$. Reinserting each repeated edge according to Corollary~\ref{cor:repeatedge} (with explicit formula given in Lemma~\ref{lem:repeatedgeMGS}) gives a generating set for $C_n$.  Because it has the same number of elements as vertices, the final output $\cB_n$ is an MGS for $C_n$ per Lemma~\ref{lem:folk}.
\end{proof}

Example~\ref{ex:mainpolysplines} shows an example of how to use the algorithm in Theorem~\ref{thm:polysplines} to produce a homogeneous MGS.  First, we give the following corollary, which classifies degree sequences for all splines on cycles whose edge-labels are principal ideals generated by homogeneous degree-two polynomials in $\Bbbk[x,y]$.

\begin{corollary}\label{cor:polysplines}
Let $G=(V,E)$ be an $n$-cycle and let $\cI$ be the set of principal ideals of $R = \mathbb{\Bbbk}[x,y]$ of the form $\lang (x+ky)^2 \rang$, where $k\in\Bbbk$. Let $\alpha\co E\ra \cI$ be an edge-labeling of $G$.  Then the following hold:
\begin{enumerate}
    \item\label{item:1} If $(G,\alpha)$ has exactly one distinct edge label, then its degree sequence is $(1,0,n-1)$.
    \item\label{item:2} If $(G,\alpha)$ has exactly two distinct edge labels, then its degree sequence is $(1,0,n-2,0,1)$.
    \item\label{item:3} If $(G,\alpha)$ has three or more distinct edge labels, then its degree sequence is $(1,0,n-3,2)$.
\end{enumerate}
\end{corollary}

\begin{proof}
We prove each of \eqref{item:1}--\eqref{item:3} separately.

\textit{Proof of \eqref{item:1}}.  This follows immediately from Theorem~\ref{thm:onelabel}: in the MGS $\{\mathbf{1}, \bI^{v_2}, \ldots, \bI^{v_{|V|}}\}$,  the trivial spline $\mathbf{1}$ is a degree-zero spline and each of the $(|V|-1)$-many $\bI^{v_i}$ is a degree-two spline.

\textit{Proof of \eqref{item:2}}.  
The proof is essentially an analysis of the MGS $\cB$ produced by Theorem~\ref{thm:2labalg} for a certain nice vertex ordering.  Since $(G,\alpha)$ has exactly two distinct edge-labels, we choose an ordering of the vertices satisfying Proposition~\ref{prop:order} by choosing the last vertex $v_n \in V$ to be any vertex incident to two edges with different labels; the vertex $v_1$ is chosen as the next vertex clockwise from $v_n$, and we continue choosing vertices $v_2, \ldots, v_{n-1}$ clockwise until all vertices have been ordered.  Without loss of generality, suppose that $\alpha(v_{n-1}v_n) = \lang \sfb \rang$ and $\alpha(v_nv_1) = \lang \sfa \rang$.

\begin{figure}[ht!]
\begin{tikzcd}[row sep = 6, column sep = 7]
 &  & & v_n \ar[drr,dash,"\lang\sfa\rang"] & &  &  \\
 & v_{n-1} \ar[urr,dash,"\lang\sfb\rang"] & & & & v_1 \ar[dr,dash,"\lang\sfa\rang \text{ or } \lang\sfb\rang"] &  \\
v_{n-2} \ar[ur,dash,"\lang\sfa\rang \text{ or } \lang\sfb\rang"] & & & & & & v_2 \ar[llllll,bend left, dotted,dash] \\
  & & & & & & \\
\end{tikzcd}
\caption{The idea of the proof of (2).}
\end{figure}
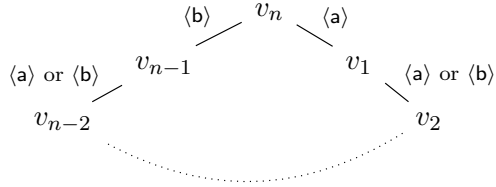

We claim that
\begin{itemize}
    \item $\bid$ is a degree-zero spline,
    \item $\bb^{i}$ is a degree-two spline for all $2 \le i < n$, and
    \item $\bb^{n}$ is a degree-four spline.
\end{itemize}
The first assertion is clear.  For the second assertion, let us assume that while producing $\bb^{i}$ using Theorem~\ref{thm:2labalg}, we chose $v_j = v_{i-1}$.
\begin{itemize}
    \item[Case 1:] The edge-label $\alpha(v_iv_{i-1}) = \lang \sfa \rang$.
    \item[Case 2:] The edge-label $\alpha(v_iv_{i-1}) = \lang \sfb \rang$.
\end{itemize}

The graph $C = (V',E')$ has vertex set $V'$ a subset of the set $\{v_i,v_{i+1},\ldots,v_n\}$ in Case 1 (resp. $\{v_i,v_{i+1},\ldots,v_{n-1}\}$ in Case 2).  In both cases, Theorem~\ref{thm:2labalg} (a) must have been applied in the production of $\bb^{i}$, so the spline $\bb^{i}$ is a degree-two spline, and we have verified the second assertion.

For the third assertion, we again assume that while producing $\bb^{n}$ using Theorem~\ref{thm:2labalg}, we chose $v_j = v_{i-1} = v_{n-1}$.  Now the graph $C$ contains the edge $v_nv_1$, so Theorem~\ref{thm:2labalg} (b) must have been applied in the production of $\bb^{n}$. Hence $\bb^{n}$ is a degree-four spline, and we have verified the third and final assertion.

\textit{Proof of \eqref{item:3}}.  
This is a consequence of Theorem~\ref{thm:polysplines}.  By Corollary~\ref{cor:repeatedge}, the homogeneous MGS $\cB_{n-k}$ for $R_{C_{n-k},\alpha_{n-k}}$ has degree sequence $(1,0,n-k-3,2)$.  For every $0 \le j \le k-1$, the homogeneous MGS $\cB_{n-k+(j+1)}$ for $R_{C_{n-k+(j+1)},\alpha_{n-k+(j+1)}}$ has degree sequence $(1,0,n-k+(j+1)-3,2)$.  After all iterations (when $j = k-1$), we obtain the homogeneous MGS $\cB_n$ with degree sequence $(1,0,n-3,2)$ as desired.
\end{proof}

\begin{example}\label{ex:mainpolysplines}
We produce a homogeneous MGS $\cB_6$ for the following edge-labeled six-cycle by using the algorithm in Theorem~\ref{thm:polysplines}.
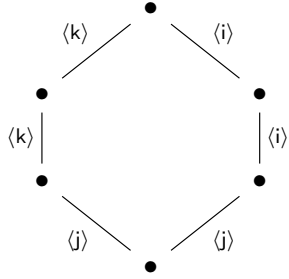
\begin{figure}[ht!]
    \centering
    \begin{tikzcd}
    & \bullet \ar[dl,dash,swap,"\lang\sk\rang"] \ar[dr,dash,"\lang\si\rang"] &\\
\bullet \ar[d,dash,swap,"\lang\sk\rang"] & & \bullet \ar[d,dash,"\lang\si\rang"] \\
\bullet & & \bullet \\
 & \bullet \ar[ul,dash,"\lang\sj\rang"] \ar[ur,dash,swap,"\lang\sj\rang"] & 
\end{tikzcd}
\caption{An edge-labeled six-cycle.}
\end{figure}

This is illustrated explicitly in Figure~\ref{fig:mainex}: one picture is shown for each reinsertion of a vertex on a repeated edge according to Corollary~\ref{cor:repeatedge}, along with the associated MGS obtained at that step via Lemma~\ref{lem:repeatedgeMGS}.

\begin{figure}[ht!]
    \centering
    \adjustbox{scale=.85}{%
    \begin{tikzcd}
    & v_1 \ar[ddl,dash,swap,"\lang\sk\rang"] \ar[ddr,dash,"\lang\si\rang"] &\\
 & &  \\
v_3 \ar[rr,dash,swap,"\lang\sj\rang"] & & v_2 \\
 &  & 
\end{tikzcd}}
\quad
${\footnotesize\cB_3 = \left\{\bid,\left(\begin{array}{l}0\\x\si\\g_1\sk\end{array}\right),\left(\begin{array}{l}0\\y\si\\g_2\sk\end{array}\right) \right\}}$

\qquad

\adjustbox{scale=.85}{%
    \begin{tikzcd}
    & v_1 \ar[dl,dash,swap,"\lang\sk\rang"] \ar[ddr,dash,"\lang\si\rang"] &\\
v_4 \ar[d,dash,swap,"\lang\sk\rang"] & &  \\
v_3 \ar[rr,dash,swap,"\lang\sj\rang"] & & v_2 \\
 &  & 
\end{tikzcd}}
\quad
${\footnotesize\cB_4 = \left\{\bid,\left(\begin{array}{l}0\\x\si\\g_1\sk\\g_1\sk\end{array}\right),\left(\begin{array}{l}0\\y\si\\g_2\sk\\g_2\sk\end{array}\right),\left(\begin{array}{l}0\\0\\0\\\sk\end{array}\right) \right\}}$

\qquad

\adjustbox{scale=.85}{%
    \begin{tikzcd}
    & v_4 \ar[dl,dash,swap,"\lang\sk\rang"] \ar[dr,dash,"\lang\si\rang"] &\\
v_3 \ar[d,dash,swap,"\lang\sk\rang"] & & v_5 \ar[d,dash,"\lang\si\rang"] \\
v_2 \ar[rr,dash,swap,"\lang\sj\rang"] & & v_1 \\
 & & 
\end{tikzcd}}
\quad
${\footnotesize\cB_5 = \left\{\bid,\left(\begin{array}{l}x\si\\g_1\sk\\g_1\sk\\0\\
0\end{array}\right),\left(\begin{array}{l}y\si\\g_2\sk\\g_2\sk\\0\\0\end{array}\right),\left(\begin{array}{l}0\\0\\\sk\\0\\0\end{array}\right),\left(\begin{array}{l}0\\0\\0\\0\\\si\end{array}\right) \right\}}$

\qquad

\adjustbox{scale=.85}{%
    \begin{tikzcd}
    & v_3 \ar[dl,dash,swap,"\lang\sk\rang"] \ar[dr,dash,"\lang\si\rang"] &\\
v_2 \ar[d,dash,swap,"\lang\sk\rang"] & & v_4 \ar[d,dash,"\lang\si\rang"] \\
v_1 & & v_5 \\
 & v_6 \ar[ul,dash,"\lang\sj\rang"] \ar[ur,dash,swap,"\lang\sj\rang"] & 
\end{tikzcd}}
\quad
${\footnotesize\cB_6 = \left\{\bid,\left(\begin{array}{l}g_1\sk\\g_1\sk\\0\\0\\
x\si\\x\si\end{array}\right),\left(\begin{array}{l}g_2\sk\\g_2\sk\\0\\0\\y\si\\y\si\end{array}\right),\left(\begin{array}{l}0\\\sk\\0\\0\\0\\0\end{array}\right),\left(\begin{array}{l}0\\0\\0\\\si\\0\\0\end{array}\right),\left(\begin{array}{l}0\\0\\0\\0\\0\\\sj\end{array}\right) \right\}}$
    \caption{An illustration of the algorithm in the proof of Theorem~\ref{thm:polysplines}.}
    \label{fig:mainex}
\end{figure}
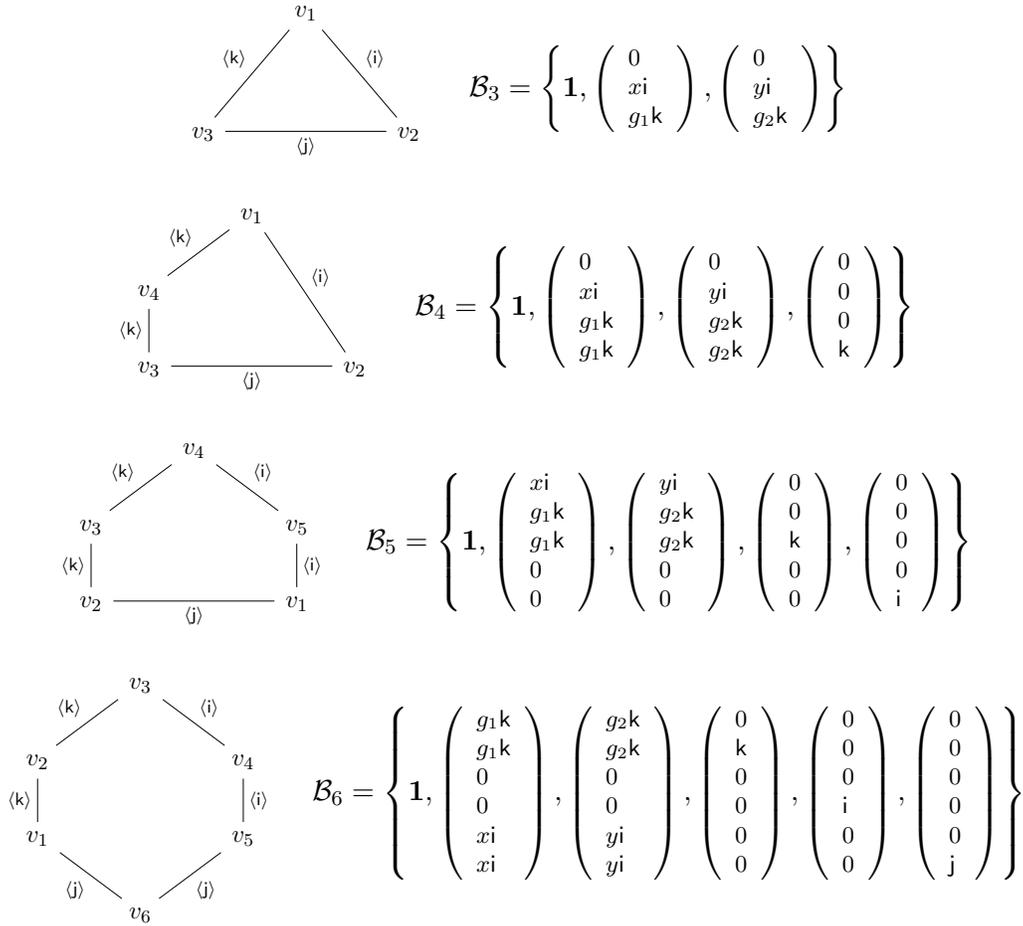
\end{example}

\section{Quotient splines and homogenization}
\label{sec: quotient splines}

In Section~\ref{sec:quotient}, we describe  {\textit{quotient splines}}, especially those that arise from a quotient map $R \rightarrow R/I$ on the coefficient ring.  The most important applications of splines involve subtle questions about quotient splines in the case when the coefficient ring $R = \Bbbk[x_1,\ldots,x_n]$, where the field is usually $\Bbbk = \mathbb{R}$ or $\mathbb{C}$.  In this case, splines inherit a notion of degree from the  degree of the polynomial at each vertex.  Classical splines usually consist of elements of degree at most $d$; we will show below that this is essentially equivalent to either of two different quotient constructions on  splines.  In GKM theory, splines model the equivariant cohomology ring of certain algebraic varieties with torus-actions; in this case, the quotient ring is isomorphic to ordinary cohomology (see Proposition~\ref{proposition: GKM and quotient splines}).   

The results in this section use the more general language of \textit{graded rings} to streamline proofs. The reader interested primarily in applications can translate everything in this section to polynomial rings as follows:
\begin{itemize}
    \item Homogeneous polynomials are those for which every nonzero term has the same degree (where the degree of a term is the sum of the exponents of all the  variables in that term).
    \item The $i^{th}$ graded part of the polynomial ring consists of $0$ together with the homogeneous polynomials of degree $i$.
    \item The edge-label $\alpha(uv)$ is homogeneous if every generator of $\alpha(uv)$ is a homogeneous polynomial.
    \item The degree of a spline $\bp$ is the maximal degree of the polynomials $\bp_v$ over all vertices $v$.
    \item A spline $\bp$ is homogeneous of degree $i$ if every nonzero $\bp_v$ is a homogeneous polynomial of degree $i$.
\end{itemize}

In this section, we establish the following:
\begin{itemize}
    \item Given any ideal $\mathcal{I} \subseteq R_{G,\alpha}$, the \textit{quotient splines} are the elements of the quotient ring $R_{G,\alpha}/\mathcal{I}$.  This quotient inherits the structure of an $R$-module.
    \item Suppose that $R$ is a graded ring. Suppose that $\alpha$ is an edge-labeling of $G$ for which every $\alpha(uv)$ is a homogeneous ideal, called a \textit{homogeneous edge-labeling}.  Then the splines $R_{G,\alpha}$ also form a graded ring whose $i^{th}$ graded part consists of splines $\bp \in R_{G,\alpha}$ for which $\bp_v$ is homogeneous of degree $i$ for each vertex $v$.
    \item Every surjective ring homomorphism $\pi\co R \rightarrow S$ induces a change-of-coefficients map on splines from $R_{G,\alpha}$ to $S_{G,\pi \circ \alpha}$.  We use this most for the usual projection map $\pi\co R \rightarrow R/I$.  
    \item Suppose that $I$ is a homogeneous ideal and that $\pi\co R \rightarrow R/I$ is the quotient map.  Assume that $\alpha$ is a homogeneous edge-labeling, and let $\mathcal{I}$ denote the splines $\bp$ with $\pi(\bp_v) \in I$ for all vertices $v$.  Then the ring of quotient splines is isomorphic to the ring of splines over the quotient ring.  In other words, we have:
\[ \left(R_{G,\alpha} \right)/\mathcal{I} \cong \left(R/I\right)_{G,\pi \circ \alpha}
\]
\end{itemize}
In Section~\ref{sec:homog}, we restrict to the case when our coefficient ring is a polynomial ring. The key point for applications is that when the coefficients are polynomials and the edge-labeling is homogeneous, the (quotient) ring of splines of degree at most $k$ is isomorphic to the ring obtained by restricting degree on the entries of the original collection of splines.
Moreover, we describe how to homogenize an edge-labeling so we can use these results for non-homogeneous edge-labelings. This will be the main tool used when we interpret the results of this paper for classical splines in Section~\ref{sec:applications}.

\subsection{Quotient splines}\label{sec:quotient}

We start with the basic definition of quotient splines as well as a natural quotient map on splines.  The underlying ideas are similar to that of Bowden and the third author \cite[Theorem 3.7]{BowTym15}.

\begin{definition} \label{definition: splines over quotient ring}
Let $\pi\co R \rightarrow S$ be a surjective ring homomorphism.  If $(G,\alpha)$ is an edge-labeled graph over $R$, then define $(G, \pi \circ \alpha)$ to be the edge-labeled graph over the ring $S$ in which each edge $e$ is labeled $\pi(\alpha(e))$. We call  $S_{G,\pi \circ \alpha}$ the \textit{splines induced by $\pi$ over $S$} or, when context is clear, the \textit{splines on $(G,\alpha)$ over $S$}.  The map $\pi$ may be referred to as a \textit{change of coefficients} for the splines.

In particular if $I$ is an ideal in $R$ and $\pi\co R \rightarrow R/I$ is the quotient map, then the elements of $\left(R/I\right)_{G,\pi \circ \alpha}$ are called \textit{splines reduced mod $I$} or just \textit{splines mod $I$}. 
\end{definition}

Not every ring homomorphism sends ideals in the domain to ideals in the codomain, but all surjective ring homomorphisms do. Thus the previous definition makes sense.

We give the following proposition for completeness; it simply confirms that the map $\pi$ induces a homomorphism from splines on $(G,\alpha)$ over $R$ to splines on the same graph over $S$.

\begin{proposition} \label{proposition: surjective mostly functoriality}
Suppose $\pi\co R \rightarrow S$ is a surjective ring homomorphism.  For each spline $\bp \in R_{G,\alpha}$ and vertex $v \in V$, the rule
\[\pi_*(\bp)_v = \pi(\bp_v)\]
defines a map $\pi_* \co R_{G,\alpha} \rightarrow S_{G,\pi \circ \alpha}$ that is a homomorphism of both graded rings and $R$-modules when $S_{G,\pi \circ \alpha}$ is endowed with the $R$-action
\[r \cdot \pi_* (\bp) := \pi(r)\pi_*(\bp).\]
\end{proposition}

\begin{proof}
The image $\pi_*(\bp)$ is a spline because for each edge $uv \in E$ we have
\[\pi(\bp_u) - \pi(\bp_v) = \pi(\bp_u - \bp_v) \in \pi (\alpha(uv)).\]
The rest of the claim follows by definition of the map $\pi_*$.
\end{proof}

Recall that a homogeneous ideal in a graded ring is characterized by the property that if $f \in I$ decomposes into homogeneous parts $f=f_0+f_1+\cdots + f_k$ then each homogeneous part $f_j \in I$ as well.  When $R$ is a graded ring and each edge is labeled by a homogeneous ideal, the ring of splines $R_{G,\alpha}$ admits a grading induced by the grading on $R^{|V|}$ as follows.

\begin{proposition} \label{proposition: grading on splines}
Suppose that $R$ is a graded ring with $R_0$ denoting the collection of degree-zero ring elements.  Further suppose that $(G,\alpha)$ is graph whose edges are all labeled by homogeneous ideals $\alpha(uv)$.

Then the ring of splines $R_{G,\alpha}$ is graded with homogeneous parts $\left( R_{G,\alpha} \right)_i$ containing precisely those splines $\bp$ for which $\bp_v$ has degree $i$ for all $v \in V$.  
\end{proposition}

\begin{proof}
Suppose that $\bp \in R_{G,\alpha}$, and that for each $v \in V$ the ring element $\bp_v \in R$ decomposes into homogeneous parts denoted
\[\bp_v = p_{0,v}+p_{1,v} + \cdots + p_{i,v}\]
with degrees $0, 1, \ldots, i$ respectively.  (Each part may be zero.)  The homogeneous part $\bp_j$ of the spline $\bp$ is defined at each vertex $v$ by
\[\bp_{j,v} = p_{j,v}.\]
We need to show that for each $j$ the homogeneous part $\bp_j$ is also a spline in $R_{G,\alpha}$.  

Suppose that $uv$ is an edge in $G$.  Expanding $\bp_u, \bp_v$ into homogeneous parts yields for $\bp_u - \bp_v$ the expression
\[\left(p_{0,u}+p_{1,u} + \cdots + p_{i,u}\right) - \left(p_{0,v}+p_{1,v} + \cdots + p_{i,v}\right),\]
which is in $\alpha(uv)$ by the spline condition on $\bp$.  Using associativity and commutativity, the above expression is equal to 
\[\left(p_{0,u}-p_{0,v}\right)+ \left( p_{1,u}-p_{1,v} \right) + \cdots + \left( p_{i,u} - p_{i,v} \right).\]
Homogeneous ideals contain each homogeneous part of each element of the ideal, so
\[p_{j,u} - p_{j,v} \in \alpha(uv)\]
for each $j$.  Thus $\bp_j$ satisfies the spline condition at each edge $uv$ for each $j$ and so $\bp_j \in R_{G,\alpha}$ for all $j$ as desired.  
\end{proof}

\begin{example} \label{example: first reason for homogeneity}
If edge-labels are not homogeneous ideals, then the ring of splines might not be graded by degree. For instance, consider the edge-labeled graph in Figure~\ref{figure: example of non-homogenized graph} with coefficient ring of polynomials in one variable.

\begin{figure}[h]
\begin{picture}(200,20)(-100,-10)
\put(-59,-2){$u$}
\put(53,-2){$v$}
\put(-50,0){\line(1,0){100}}
\put(-17,5){$\langle x^2-1 \rangle$}
\end{picture}
\caption{Example of an edge-labeled graph with non-homogeneous polynomial labeling.} \label{figure: example of non-homogenized graph}
\end{figure}
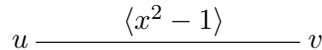
Consider the spline $\bp$ on this edge-labeled graph with $\bp_u = 0$ and $\bp_v = x^2-1$.  Under the typical grading on polynomials, the homogeneous parts of $\bp$ would be the splines $(0,-1)$ in degree zero and $(0,x^2)$ in degree two.  However neither of these are in $R_{G,\alpha}$.
\end{example}

Proposition~\ref{proposition: grading on splines} means that we can define the following.

\begin{definition}\label{definition: homogeneity and degree d parts}
    An edge-labeling $\alpha$ is \textit{homogeneous} if $\alpha(uv)$ is a homogeneous ideal for every edge $uv$.  The homogeneous degree-$i$ part of the ring $R_{G,\alpha}$ with grading defined in Proposition~\ref{proposition: grading on splines} is denoted $\left( R_{G,\alpha} \right)_i$ and called the collection of \textit{(homogeneous) splines of degree $i$}. We denote the $R_0$-submodule of \textit{splines of degree at most} $d$ by 
    \[\mathcal{S}_d(G,\alpha) = \left( R_{G,\alpha} \right)_0 \oplus \left( R_{G,\alpha} \right)_1 \oplus \cdots \oplus \left( R_{G,\alpha} \right)_d\]
\end{definition}

The next result connects  quotient splines with splines over the quotient ring.  We will be most interested in the case when $I_{k+1} \subseteq R$ consists of all polynomials of degree at least $k+1$, in which case the following result will compare all splines of degree at most $k$ to the splines whose coefficients are reduced mod $I_{k+1}$. 

\begin{proposition} \label{proposition: quotient splines and reducing coeffs}
Suppose that $R$ is a graded ring and denote the collection of ring elements of degree zero by $R_0$.  Suppose that $(G,\alpha)$ is an edge-labeled graph and that $\alpha$ is a homogeneous edge-labeling.  

Let $I$ be a homogeneous ideal with quotient map $\pi\co R \rightarrow R/I$, and let $\mathcal{I} \subseteq R_{G,\alpha}$ be the collection of splines $\bp$ with $\bp_v \in I$ for all vertices $v$.  

Then $\mathcal{I}$ forms a homogeneous ideal in the graded ring of splines $R_{G,\alpha}$, and there is a unique ring homomorphism $\rho'$ so that the following diagram commutes: 
\[\begin{tikzcd}
    R_{G,\alpha} \arrow[r, "\pi_*"] \arrow[d, "\rho"] & \left(R/I\right)_{G,\pi \circ \alpha} \\
    \left(R_{G,\alpha}\right)/\mathcal{I} \arrow[ur,
    "\rho'"] & 
\end{tikzcd}\]
Moreover the ring homomorphism $\rho'$ is injective and graded.
\end{proposition}

\begin{proof}
First we confirm that $\mathcal{I}$ forms an ideal, and in fact that $\mathcal{I} = \ker \pi_*$.  Note  that the projection  
\[\begin{array}{rl}
\pi_v \co R_{G,\alpha} &\rightarrow R \\
\bp &\mapsto \bp_v
\end{array}\]
is a surjective ring homomorphism for each vertex $v$.  Indeed, the principal ideal $R{\bf 1}$ generated by the trivial spline surjects onto $R$ at each vertex. Moreover the following diagram commutes for each $v$ since $\pi_*$ and $\pi_v$ are both defined vertex-wise:
\[\begin{tikzcd}
R_{G,\alpha} \arrow[r,"\pi_*"] \arrow[d,"\pi_v"] & (R/I)_{G,\pi \circ \alpha} \arrow[d,"\pi_v"] \\
R \arrow[r,"\pi"] & R/I
\end{tikzcd}
\]
Note that $\mathcal{I}$ is defined as the intersection of the preimage of $I$ under all of the $\pi_v$, namely:
\[ \mathcal{I} =  \displaystyle \bigcap_{v \in V} \pi_v^{-1}(I).\]
This means $\mathcal{I}$ is the intersection of a finite number of ideals, and thus an ideal.  It also shows
\[ \pi \circ \pi_v(\mathcal{I}) = \pi_v \circ \pi_*(\mathcal{I}) = \{0\}\]
for all $v \in V$.  By definition the only spline $\bq \in (R/I)_{G,\pi \circ \alpha}$ with $\pi_v(\bq) = 0$ for all $v \in V$ is the zero spline.  Thus $\pi_*(\mathcal{I}) = \{\bf{0}\}$.  Conversely every such $\bq$ satisfies
\[\pi_v(\bq) \in \ker \pi = I\]
for all $v \in V$.  It follows that $\ker \pi_* = \mathcal{I}$ as desired.

Next we show that $\mathcal{I}$ is homogeneous in $R_{G,\alpha}$.  Proposition~\ref{proposition: grading on splines} showed that $R_{G,\alpha}$ is a graded ring and that its graded parts are defined by 
\[ \left(R_{G,\alpha}\right)_i = \bigcap_{v \in V} \pi_v^{-1}\left(R_i\right), \]
where $R_i$ is the $i^{th}$ graded part of $R$ for each $i$.  Write $\bp \in \mathcal{I}$ as a sum of homogeneous parts as
\[\bp = \bp_0 + \bp_1 + \bp_2 + \cdots + \bp_d,\]
where each $\bp_i$ is homogeneous of degree $i$ in $R_{G,\alpha}$.  We want to show that $\bp_i$ is in $\mathcal{I}$ for each $i$.  Consider $\pi_v(\bp)$ for each vertex $v$.  Since $\pi_v$ is a ring homomorphism, we have
\[\bp_v = (\bp_0)_v + (\bp_1)_v + \cdots + (\bp_d)_v.\]
By definition of $\mathcal{I}$, we know that $\bp_v \in I$ for all $v$.  By definition of the grading on splines, we know $(\bp_i)_v \in R_i$ is homogeneous of degree $i$ for all $0 \leq i \leq d$ and all $v$.  Since the ideal $I$ is homogeneous, we conclude that $(\bp_i)_v \in I$ for all $v$ and all $i$. Thus $\pi_v(\bp_i) \in I$ for all $v$ and all $i$.  In other words, we have $\bp_i \in \mathcal{I}$ for all $i$, so by definition $\mathcal{I}$ is homogeneous.

The grading induced on quotient rings by homogeneous ideals $(R_{G,\alpha})/\mathcal{I}$ is defined by
\[\left( (R_{G,\alpha})/\mathcal{I}\right)_i = \left(R_{G,\alpha}\right)_i + \mathcal{I}/\mathcal{I}.\]
Thus our proof of homogeneity also shows that 
\[\rho'(\bp_i+\mathcal{I}) = \pi_*(\bp_i).\]
Since the restriction of the spline $\pi_*(\bp_i)$ to each vertex $v$ is given by $(\bp_i)_v + I$ and $(\bp_i)_v$ is by construction homogeneous of degree $i$ in $R$, we conclude that $\pi_*(\bp_i)$ is homogeneous of degree $i$ according to the grading on splines over $R/I$ from Proposition~\ref{proposition: grading on splines}.

Thus $\rho'$ is an injective, graded ring homomorphism as desired.  
\end{proof}

It is natural to wonder when the map $\rho'$ is surjective.  The next lemma gives a condition that may seem contrived.  Indeed the only example we know of where it is useful is when $R$ is a polynomial ring, $I_{d+1} \subseteq R$ is the ideal generated by all monomials of degree at least $d+1$, and $\mathcal{I}_{d+1}$ is generated by all homogeneous splines of degree at least $d+1$.  (Indeed, this setting will be the main one in our Section~\ref{sec:applications} on applications.) 
However, the general language of the lemma makes the proof transparent.

\begin{lemma} \label{lemma: surjective map}
    Suppose $R$, $\alpha$, $I$, and $\mathcal{I}$ all satisfy the hypotheses of Proposition~\ref{proposition: quotient splines and reducing coeffs}.  Suppose that each graded part $R_i$ of the coefficient ring can be written as a direct sum
    \[R_i = J_i \oplus (R_i \cap I)\]
    for some $R_0$-module $J_i$ and that this extends to the $R_0$-submodule $\alpha(uv)$ for each edge $uv$ via
    \[\alpha(uv) \cap R_i = (\alpha(uv) \cap J_i) \oplus (\alpha(uv) \cap R_i \cap I)\]
    for all $i$.  Then the map $\pi_*\co R_{G,\alpha} \rightarrow (R/I)_{G,\pi \circ \alpha}$ is surjective and hence the map 
    \[\rho'\co (R_{G,\alpha})/\mathcal{I} \rightarrow (R/I)_{G,\pi \circ \alpha}\] is an isomorphism of graded rings.
\end{lemma}

\begin{proof}
Suppose that $\bq \in (R/I)_{G,\pi \circ \alpha}$ and assume that $\bq$ is homogeneous of degree $i$.  We will show that $\bq$ is in the image of $\pi_*$.  Surjectivity for general $\bq$ will then follow, since $(R/I)_{G,\pi \circ \alpha}$ is a graded ring and so every spline $\bq$ can be expressed as a sum of homogeneous splines, each of which we will have shown to be in the image of $\pi_*$.  Since $\pi_*$ is a homomorphism, we will thus conclude that every spline $\bq \in (R/I)_{G,\pi \circ \alpha}$ is in the image of $\pi_*$ as desired.

Assume $\bq$ is a homogeneous spline of degree $i$ in $(R/I)_{G,\pi \circ \alpha}$.  We first show that for each $v \in V$ the inverse image satisfies
\[\pi^{-1}(\bq_v)  = \bp_v + I\]
for a unique polynomial $\bp_v \in J_i$.  To do this, recall that the $i^{th}$ graded part of the quotient $R/I$ is $R_i+I/I$.  Thus $\bq_v \in R_i+I/I$ implies there is a polynomial $f_v \in R_i$ with 
\[\pi^{-1}(\bq_v) = f_v + I.\]
Now let $\bp_v$ be the unique element of $J_i$ with 
\[\bp_v - f_v \in I\]
guaranteed by our hypothesis $R_i = J_i \oplus (R_i \cap I)$.

We next show that $\bp$ is in fact a spline.  (If so, it is by construction homogeneous of degree $i$.) For each edge $uv$, note that 
\[\pi(\bp_u) - \pi(\bp_v) = \bp_u - \bp_v + I\]
by construction of $\bp$.  Since $\bq$ is a spline on the quotient ring, there is $g \in I$ for which 
\[\bp_u - \bp_v + g \in \alpha(uv).\]
Since $\alpha(uv)$ is homogeneous, we may assume $g$ is homogeneous of degree $i$ without loss of generality.  Now by the assumption that $\alpha(uv) \cap R_i$ respects the direct-sum decomposition, we conclude 
\[\bp_u - \bp_v \in \alpha(uv)\]
as well.  This is true for all edges $uv$, so $\bp \in R_{G,\alpha}$ and $\pi_*(\bp) = \bq$.  

Thus we have proven our claim that $\pi_*$ is surjective.  Moreover the map $\pi_*$ is the composition 
\[\pi_* = \rho' \circ \rho\]
and $\rho$ is surjective because it is a quotient map.  So $\rho'$ must also be surjective.  Proposition~\ref{proposition: quotient splines and reducing coeffs} showed that $\rho'$ is an injective homomorphism, so it is an isomorphism.
\end{proof}

\begin{corollary} \label{corollary: splines over quotients}
Suppose that $R$ is a graded ring and $(G,\alpha)$ an edge-labeled graph with homogeneous edge-labeling.
Suppose that $I_{d+1} \subseteq R$ is the ideal generated by all homogeneous elements of degree at least $d+1$, namely
\[I_{d+1} = R_{d+1} \oplus R_{d+2} \oplus R_{d+3} \oplus \cdots,\]
and that $\mathcal{I}_{d+1} \subseteq R_{G,\alpha}$ is the ideal of splines generated by the homogeneous splines of degree at least $d+1$.  Then the quotient splines $(R_{G,\alpha})/\mathcal{I}_{d+1}$ are isomorphic to the ring of splines reduced mod $I_{d+1}$.  Furthermore, taking $\mathcal{S}_d(G,\alpha)$ to be the $R_0$-submodule of splines of degree at most $d$ as in Definition~\ref{definition: homogeneity and degree d parts}, we have a commutative diagram 
\[\begin{tikzcd}
    \mathcal{S}_d(G,\alpha) \arrow[r, "\pi_*"] \arrow[d, "\rho"] & \left(R/I_{d+1}\right)_{G,\pi \circ \alpha} \\
    \left(R_{G,\alpha}\right)/\mathcal{I}_{d+1} \arrow[ur,
    "\rho'"] & 
\end{tikzcd}\]
in which all maps are graded $R_0$-module isomorphisms and $\rho'$ is an isomorphism of graded rings.
\end{corollary}

\begin{proof}
Whenever $I$ is a homogeneous ideal in a graded ring $S$, the quotient $S/I$ is a graded ring with graded parts
\[S/I = \bigoplus_{i \geq 0} (S_i+I)/I.\]
Suppose $I$ is generated by all homogeneous elements of degree at least $d+1$.  Then the intersection $S_i \cap I$ is empty or all of $S_i$, depending on whether $i < d+1$ or not.  So the quotient $(S_i+I)/I$ is either isomorphic to $S_i$ or $\{0\}$ as an $R_0$-module depending on whether $i<d+1$ or not.  We conclude that as additive groups and as $R_0$-modules, 
\begin{equation} \label{equation: quotient cong to subspace}
S/I \cong \bigoplus_{i=0}^d S_i.
\end{equation}

When $S = R_{G,\alpha}$ and $I = \mathcal{I}_{d+1}$, this gives us the isomorphism $\rho$ in the commutative diagram.  When $S=R$ and $I=I_{d+1}$, this gives us an isomorphism 
\[R/I_{d+1} \cong R_0 \oplus R_1 \oplus R_2 \oplus \cdots \oplus R_d.\]

The ideal $I_{d+1}$ satisfies the hypotheses of Lemma~\ref{lemma: surjective map} vacuously since 
\[R_i = R_i \oplus \{0\}\]
and $I_{d+1} \cap R_i$ is one of those two summands for each $i$.  For the same reason, the edge-labeling $\alpha$ also satisfies the hypotheses of Lemma~\ref{lemma: surjective map}.  Hence we conclude that $\rho'$ is a (graded) isomorphism of graded rings.  Since $\pi_* = \rho' \circ \rho$, this completes our proof.
\end{proof}

\begin{example} \label{remark: homogeneity assumption is needed}
We note that the hypothesis of homogeneity is also essential in Corollary~\ref{corollary: splines over quotients}.  To see this, we continue our analysis of the graph in Figure \ref{figure: example of non-homogenized graph}.

Consider the spline condition over the edge labeled $\langle x^2-1 \rangle$.  When we take coefficients in $\mathbb{C}[x]/I_2$, the edge-label becomes $-1$ and so all vertex-labelings satisfy the spline condition.  However, zero is the only polynomial with degree at most one that is divisible by $x^2-1$.  This means that only constant splines are in the image of the map from $\mathcal{S}_1(G,\alpha)$ to $(R/I_2)_{G,\pi \circ \alpha}$. 
\end{example}

The next result uses Corollary~\ref{corollary: splines over quotients} to show that the image of a homogeneous MGS under the quotient map is again a homogeneous MGS.  

\begin{corollary} \label{corollary: degree sequence of quotient MGS}
Let $R$ be a graded ring, and let $(G,\alpha)$ be an edge-labeled graph with homogeneous edge-labeling $\alpha$. Fix a nonnegative integer $d$.  Let $I_{d+1}$ be the ideal in $R$ generated by all homogeneous elements of degree at least $d+1$, and let $\mathcal{I}_{d+1}$ be the ideal in $R_{G,\alpha}$ generated by all homogeneous splines of degree at least $d+1$. 

Suppose $\mathcal{B}$ is an MGS for $R_{G,\alpha}$ consisting of homogeneous splines. Then the nonzero elements in $\pi_*(\mathcal{B})$ form a homogeneous MGS for $(R/I_{d+1})_{G,\pi \circ \alpha}$, and the degree sequence of $\pi_*(\mathcal{B})$ consists of the first $d+1$ terms of the degree sequence of $\mathcal{B}$.
\end{corollary}

\begin{proof}
Write $\mathcal{B}_d$ for the elements of $\mathcal{B}$ that have degree at most $d$.  Every element of $\mathcal{B}$ is homogeneous, so the $R$-linear combinations of elements of $\mathcal{B}$ that are not in $\mathcal{B}_d$ have degree at least $d+1$ and hence are in the set-theoretic complement of $\mathcal{S}_d(G,\alpha)$. Thus  $\mathcal{B}_d$ generates $\mathcal{S}_d(G,\alpha)$ and hence $\rho(\mathcal{B}_d)$ generates $(R_{G,\alpha})/\mathcal{I}_{d+1}$ under the isomorphism $\rho$ of Corollary~\ref{corollary: splines over quotients}.

If $\mathcal{B}_d$ were not an MGS for $\mathcal{S}_d(G,\alpha)$, then we could find an MGS $\mathcal{B}' \subseteq \mathcal{S}_d(G,\alpha)$ of strictly smaller cardinality than $\mathcal{B}_d$ that also generated $\mathcal{S}_d(G,\alpha)$.  Replacing the elements of $\mathcal{B}_d$ in $\mathcal{B}$ with those of $\mathcal{B}'$ would give a generating set for $R_{G,\alpha}$ of strictly smaller cardinality than $\mathcal{B}$.  This contradicts the hypothesis that $\mathcal{B}$ is an MGS.

Finally, the degrees of $\mathcal{B}_d$ in $\mathcal{S}_d(G,\alpha) \subseteq R_{G,\alpha}$ are the first $d+1$ terms of the degree sequence of $\mathcal{B}$ by definition of $\mathcal{B}_d$.  Since the maps $\rho$ and $\rho'$ are degree-preserving isomorphisms, we know that the same is true for $\rho(\mathcal{B}_d)$ and $\pi_*(\mathcal{B}_d)$ as well.  This proves the claim.
\end{proof}

\subsection{Splines over polynomial rings and homogenization}\label{sec:homog}
The case most relevant to classical splines is when $R$ is a polynomial ring and $I_{k+1}$ is the ideal generated by all monomials of degree at least $k+1$. This is also the case on which our examples focused.  

What we address in this section is non-homogeneous edge-labelings, which are typical in applications of classical splines. We also describe \textit{homogenization}, an algebraic process to transform a non-homogeneous polynomial into a homogeneous polynomial of the same degree by inserting a new variable.  Our main result shows that the process of homogenization creates an isomorphic module of splines over the original polynomial ring.  Thus all of our results about module bases---including minimal generating sets---hold in a reasonable sense for non-homogeneous splines.  The main subtlety is that the dimension  as a \textit{vector space} may change if we do not carefully keep track of the additional variable.

{\bf In this section, we use $\Bbbk$ to denote a field of characteristic zero, in practice usually $\mathbb{R}$ or $\mathbb{C}$, and denote the polynomial ring \[R^n := \Bbbk[x_1,x_2,\ldots,x_n].\]}

\begin{remark}
    If $\Bbbk$ had finite characteristic, then the definitions in this section would still make sense but specific analyses would be more complicated; for instance, the kernel of the evaluation map $e^{n+1}$ defined in Proposition~\ref{proposition: evaluation homomorphism on polynomial rings} below would be larger.  Understanding the case of finite characteristic remains an open question.
\end{remark}

We start by defining underlying terminology and then define homogenization of polynomials and splines.  Homogenizing polynomials is common in algebraic geometry, where it is used to associate a projective variety to an affine variety; see, e.g., \cite[Section 8.4]{CLO15} or \cite[Section 3.3]{SKKT00}.  For examples of the use of homogenization of polynomials in the theory of splines see, e.g., \cite{BilRos91} or \cite{DS20}.

\begin{definition}\label{definition: homogenization}
Monomials in $R^n$ are in bijective correspondence with vectors $\Vec{v} \in \mathbb{N}^n$ where $\mathbb{N}$ denotes the nonnegative integers, according to the rule $x^{\Vec{v}} = x_1^{v_1}x_2^{v_2} \cdots x_n^{v_n}$.
  The \textit{degree} of a monomial $x^{\Vec{v}}$ in $R^n$ is
    \[\deg x^{\Vec{v}} = \sum_{i=1}^n v_i.\]
    The \textit{degree} of a polynomial $f \in R^n$ is the maximal degree of its nonzero monomials; namely, if $f(x_1, \ldots, x_n) = \sum_{\Vec{v} \in \mathbb{N}^n} c_{\Vec{v}}x^{\Vec{v}}$ with only a finite number of nonzero coefficients $c_{\Vec{v}} \in \Bbbk$ then 
    \[ \deg(f) = \max_{c_{\Vec{v}} \neq 0} \{\deg x^{\Vec{v}}\}.\]
Given a polynomial $f \in R^n$,
 its \textit{homogenization} $\widetilde{f} \in R^{n+1}$ is the polynomial
\[\widetilde{f}(x_1, \ldots, x_n, x_{n+1}) = \sum_{\Vec{v} \in \mathbb{N}^n} c_{\Vec{v}}x^{\Vec{v}}x_{n+1}^{\deg(f) - \deg(x^{\Vec{v}})} = \sum_{\scriptsize \begin{array}{cc} \Vec{v} \in \mathbb{N}^{n}: \\ \sum_i v_i \leq \deg f \end{array}} c_{\Vec{v}}x^{(v_1, v_2, \ldots, v_n, \deg(f) - (\sum_i v_i))}.\] 
We note that if $f$ were already homogeneous, then the exponent of $x_{n+1}$ in the middle of the above displayed equation is zero; thus, homogeneous $f$ are unchanged by homogenization. The homogenization of an ideal $I$ of $R^n$ is the ideal generated by the homogenizations of all elements of $I$.

    The degree of a spline $\bp \in R^n_{G,\alpha}$ was defined in the fourth bullet point of Section~\ref{sec: quotient splines} (for the current case under consideration, i.e. when the coefficient ring is a polynomial ring).
    Given a spline $\bp \in R^n_{G,\alpha}$ over the polynomial ring, its \textit{homogenization} $\widetilde{\bp}$ is defined vertex-wise by
\[(\widetilde{\bp})_v = \widetilde{(\bp_v)}x_{n+1}^{\deg(\bp) - \deg(\bp_v)}.\]
Given an edge-labeling $\alpha$ for the graph $G$, its \textit{homogenization} $\widetilde{\alpha}$ is defined at all edges $uv$ by the rule that $\widetilde{\alpha}(uv)$ is the smallest ideal containing \[\{\widetilde{f} \mid f \in \alpha(uv)\};\]
equivalently, $\widetilde{\alpha}$ is obtained by homogenizing all of the edge-labeling ideals.
\end{definition}

\begin{remark} 
The homogenizations of polynomials, splines, and edge-labelings are similar but have several differences.
\begin{itemize}
    \item The homogenization $\widetilde{f}$ of any polynomial is a homogeneous polynomial of the same degree as $f$. Indeed, the homogenization of a polynomial $f$ of degree $d$ can be defined as
\[\widetilde{f}(x_1,x_2,\ldots,x_{n+1}) = x_{n+1}^d f(x_1/x_{n+1}, x_2/x_{n+1}, \ldots, x_n/x_{n+1}).\]
\item The degree of the homogenization $\widetilde{\bp}$ is the degree of the original spline $\bp$ but the homogenization of a spline $\bp$ may change the degree of some of the polynomials $\bp_u$.  
\item Note that the homogenization of any ideal is a homogeneous ideal since it can be generated by homogeneous polynomials. This means that $\widetilde{\alpha}$ is a homogeneous edge-labeling, as our terminology suggests.
\end{itemize}
\end{remark}

\begin{example}
The polynomial $x^2-1$ has degree two, so its homogenization is $x^2-y^2$. Consider the (non-homogeneously) edge-labeled graph $(G,\alpha)$ of Figure~\ref{figure: example of non-homogenized graph}.  The polynomials $\bp_u = x$ and $\bp_v = x+x^2-1$ define a spline $\bp \in R^2_{G,\alpha}$.  Note that $\bp_u$ is homogeneous, so its homogenization is $\widetilde{\bp_u} = x = \bp_u$.  However the homogenization $\widetilde{\bp}$ of the spline has $\widetilde{\bp}_u = xy$ and $\widetilde{\bp}_v = xy+x^2-y^2$.
\end{example}

 For polynomials over an integral domain, degree respects multiplication in the sense that
 \[\deg(fg) = \deg(f)+\deg(g)\]
 for all polynomials $f, g$.  It follows that homogenization respects multiplication in the sense that
 \[\widetilde{fg} = \widetilde{f}  \widetilde{g}.\]
However homogenization is not always additive since, e.g. 
$\widetilde{x^2+y} + \widetilde{-x^2+x}$
has degree two so is not the homogenization of the sum $x^2+y-x^2+x = y+x$.

Nonetheless we have an evaluation map $e^{n+1}\co R^{n+1} \rightarrow R^n$ that is both a ring homomorphism and the inverse map of homogenization, as follows. We state the following result for convenience; its proof can be found in undergraduate algebra texts.

\begin{proposition} \label{proposition: evaluation homomorphism on polynomial rings}
    Let $e^{n+1}\co R^{n+1} \rightarrow R^n$ be the evaluation ring homomorphism defined by 
        \[ e^{n+1}(x_i) = \left\{ \begin{array}{rl} x_i & \textup{ if } 1 \leq i \leq n, \\ 1 & \textup{ if } i=n+1. \end{array} \right.\] 
Then $e^{n+1}$ is surjective, inverts the homogenization map in the sense that
    \[e^{n+1} \left( \widetilde{f} \right) = f, \]
    and is degree-preserving on homogenized polynomials in the sense that
    \[\deg\left(e^{n+1}(\widetilde{f})\right) = \deg(\widetilde{f})=\deg(f).\]
    In particular, if $f \in (R^n)_d \subseteq R^{n+1}$ is homogeneous of degree $d$, then $e^{n+1}(f)=f$.
    Moreover the kernel of the map $e^{n+1}$ is the principal ideal
    \[\ker e^{n+1} = (x_{n+1}-1)R^{n+1}. \]
\end{proposition}

In particular, observe that the only homogeneous element of $\ker e^{n+1}$ is $0$.

We now show that the homogenization of a spline is in fact a spline on the homogenized edge-labeling and describe the induced map $e^{n+1}_*$ from Proposition~\ref{proposition: surjective mostly functoriality} in this case.

\begin{proposition} \label{proposition: evaluation map is surjective}
For each spline $\bp \in R^n_{G,\alpha}$ the homogenization $\widetilde{\bp}$ is a spline in $R^{n+1}_{G,\widetilde{\alpha}}$. The ring homomorphism induced by $e^{n+1}$ on the homogenized splines $R^{n+1}_{G,\widetilde{\alpha}}$ has image $R^n_{G,\alpha}$, namely
\[e^{n+1}_*\co R^{n+1}_{G,\widetilde{\alpha}} \rightarrow R^n_{G,\alpha}\]
is surjective.  Moreover $e^{n+1}_*$ inverts the homogenization map in the sense that 
\[e^{n+1}_*(\widetilde{\bp}) = \bp,\]
and $e^{n+1}_*$ preserves degree of homogenized splines in the sense that
\[\deg(e^{n+1}_*(\widetilde{\bp})) = \deg(\widetilde{\bp}) = \deg(\bp).\]
The kernel of $e^{n+1}_*$ is the ideal $(x_{n+1}-1) R^{n+1}_{G,\widetilde{\alpha}}$ inside the ring $R^{n+1}_{G,\widetilde{\alpha}}$.
\end{proposition}

\begin{proof}
First we show that $e^{n+1} \circ \widetilde{\alpha} = \alpha$ for each edge. If $f \in \alpha(uv)$ then by definition of the homogenized edge-labeling, we know $\widetilde{f} \in \widetilde{\alpha}(uv)$.   Proposition~\ref{proposition: evaluation homomorphism on polynomial rings} says $e^{n+1}(\widetilde{f}) =f$ so we conclude $e^{n+1}(\widetilde{\alpha}(uv)) \supseteq \alpha(uv)$ for all edges $uv$.  Now we show the opposite inclusion.  Let $\sum g_i \widetilde{f_i} \in \widetilde{\alpha}(uv)$ be arbitrary, in the sense that $g_i \in R^{n+1}$ and $f_i \in \alpha(uv)$ are arbitrary for each $i$.  Then 
\[e^{n+1}\left(\sum g_i \widetilde{f_i}\right) = \sum e^{n+1}(g_i) f_i\]
by Proposition~\ref{proposition: evaluation homomorphism on polynomial rings}.  Since $e^{n+1}\co R^{n+1} \rightarrow R^n$ is surjective, we conclude that $e^{n+1}(\widetilde{\alpha}(uv)) \subseteq \alpha(uv)$ for all edges $uv$ and so $e^{n+1} \circ \widetilde{\alpha} = \alpha$ as desired.

Next we show that the induced map $e^{n+1}_*$ is surjective.  Our strategy is to show that for each spline $\bp \in R^n_{G,\alpha}$ the homogenization $\widetilde{\bp}$ is a spline in $R^{n+1}_{G,\widetilde{\alpha}}$.  For each edge $uv$, the difference $\widetilde{\bp}_u - \widetilde{\bp}_v $ is a homogeneous polynomial.  Moreover the difference
\[\bp_u - \bp_v = f \in \alpha(uv),\]
since $\bp \in R^n_{G,\alpha}$ is a spline.  If $f = 0$ then $\bp_u = \bp_v$, so $\widetilde{\bp}_u = \widetilde{\bp}_v$ and the spline condition is satisfied at edge $uv$. Thus assume $f \neq 0$.  We know $\deg f \leq \max \{ \deg \bp_u, \deg \bp_v\}$ which means there is a nonnegative $i \geq 0$ so that
\[ (\widetilde{\bp}_u - \widetilde{\bp}_v) - x_{n+1}^i \widetilde{f}\]
is homogeneous in $R^{n+1}$.  By construction we also know
that
\[ e^{n+1}\left((\widetilde{\bp}_u - \widetilde{\bp}_v) - x_{n+1}^i \widetilde{f}\right) = e^{n+1}\left(\widetilde{\bp}_u - \widetilde{\bp}_v\right) - e^{n+1}\left(x_{n+1}^i \widetilde{f}\right) = f - f = 0.\]
Proposition~\ref{proposition: evaluation homomorphism on polynomial rings} implies that the only homogeneous polynomial in $\ker e^{n+1}$ is $0$, so we conclude that 
\[ \widetilde{\bp}_u - \widetilde{\bp}_v =  x_{n+1}^i \widetilde{f} \in \widetilde{\alpha}(uv)\]
as desired.  Thus $\widetilde{\bp}$ is a spline in  $R^{n+1}_{G,\widetilde{\alpha}}$ and $e^{n+1}_*$ is surjective.

For the rest of the proof, we use the fact that $e^{n+1}_*$ is defined vertex-wise and so commutes with the projection $\pi_v$ to each vertex in the sense that 
\[\pi_v \circ e^{n+1}_* = e^{n+1} \circ \pi_v.\]
Since $e^{n+1}$ inverts the homogenization map on polynomials it follows that $e^{n+1}_*$ does, too, by commuting with $\pi_v$.  When homogenizing splines, there is at least one vertex $v$ with $
\pi_v\left(\widetilde{\bp}\right) = \widetilde{\pi_v(\bp)}$,
and for all other vertices $u$ there is an integer $d_u \geq 0$ with
\[\pi_u\left(\widetilde{\bp}\right) =  x_{n+1}^{d_u} \widetilde{\pi_u(\bp)}.\]
We conclude that for this vertex $v$ we have
\[\deg\left(\pi_v(\widetilde{\bp})\right) = \deg(\widetilde{\pi_v(\bp})) = \deg(\pi_v(\bp)),\]
and for all other vertices $u$ we have 
\[\deg\left(\pi_u(\widetilde{\bp})\right) \geq \deg(\pi_u(\bp)).\]
Since $\widetilde{\bp}$ is homogeneous, the definition of degree of splines implies that
\[\deg \widetilde{\bp} = \deg e^{n+1}_*(\widetilde{\bp}) = \deg \bp\]
so $e^{n+1}_*$ preserves degree of homogenized splines.

Finally we confirm the kernel of $e^{n+1}_*$ is as claimed.  Suppose that $\bp \in \ker e^{n+1}_*$.  Since $\widetilde{\alpha}$ is a homogeneous edge-labeling, we can write \[\bp = \bp_0 + \bp_1 + \cdots + \bp_d\] as a sum of splines $\bp_i \in R^{n+1}_{G,\widetilde{\alpha}}$ that are homogeneous of degree $i$ (though not necessarily in the kernel).   Now consider the sum of splines 
\[\bq := \bp + \sum_{i=0}^{d-1}\left(x_{n+1}^{d-i} - 1\right) \bp_i = \sum_{i=0}^{d} x_{n+1}^{d-i} \bp_i.\]
Note that $\bq$ is 
\begin{itemize}
    \item in the kernel of $e^{n+1}_*$ because each summand on the left-hand side is, and
    \item is homogeneous of degree $d$ because each summand on the right-hand side is.
\end{itemize}
Again use the commuting maps 
\[\pi_v \circ e^{n+1}_* = e^{n+1} \circ \pi_v\]
to conclude that for each vertex, the image $\pi_v(\bq)$ is a homogeneous polynomial of degree $d$ that is in the kernel of $e^{n+1}$.  The only homogeneous polynomial in the kernel of $e^{n+1}$ is $0$, so the image $\pi_v(\bq) = 0$ for all vertices $v$.  Thus the spline $\bq = 0$ and hence 
\[\bp =  \sum_{i=0}^{d-1}\left(1 - x_{n+1}^{d-i} \right) \bp_i =  \sum_{i=0}^{d-1}\left((1 - x_{n+1})(1+x_{n+1} + x_{n+1}^2 + \cdots + x_{n+1}^{d-i-1}) \right) \bp_i,\]
which is in $(x_{n+1}-1)R^{n+1}_{G,\widetilde{\alpha}}$.  Conversely, direct computation shows that
\[e^{n+1}_*\left((x_{n+1}-1)\bq\right) = (1-1)e^{n+1}_*(\bq) = 0\]
for all $\bq \in R^{n+1}_{G,\widetilde{\alpha}}$, so the kernel is as claimed.  This completes the proof.
\end{proof}

This gives our main result: that the map $e^{n+1}_*$  restricts to a degree-preserving $R^n$-module isomorphism between (non-homogenized) splines over $R^n$ and a natural $R^n$-submodule of the homogenized splines over $R^{n+1}$.

\begin{corollary} \label{theorem: homogenizing splines}
Suppose that $\mathcal{B}$ is a homogeneous MGS in $R^{n+1}_{G,\widetilde{\alpha}}$, and let $R^n(\mathcal{B})$ denote 
the collection of linear combinations of $\mathcal{B}$ with coefficients in $R^n$.

The restricted map $e^{n+1}_*\co R^n(\mathcal{B}) \rightarrow R^n_{G,\alpha}$ is an  $R^n$-module isomorphism.  In particular, the number of elements in a homogeneous MGS for $R^{n+1}_{G,\widetilde{\alpha}}$ is the same as the number of elements in an MGS for $R^n_{G,\alpha}$.

Suppose, in addition, that for each edge $uv$ the ideal $\alpha(uv)$ is principal.  Then $e^{n+1}_*$  preserves degree of the generators $\mathcal{B}$.  
\end{corollary}

\begin{proof}
We begin by showing that $e^{n+1}_*\left( R^n(\mathcal{B})\right)$ contains all of $R^n_{G,\alpha}$. Proposition~\ref{proposition: evaluation map is surjective} showed that $e^{n+1}_*\co R^{n+1}_{G,\widetilde{\alpha}} \rightarrow R^n_{G,\alpha}$ is surjective, so for each $\bp \in R^n_{G,\alpha}$ we may pick an element $\bq \in \left(e^{n+1}_*\right)^{-1}(\bp)$. Every spline in $R^{n+1}_{G,\widetilde{\alpha}}$ can be written as an $R^{n+1}$-linear combination  of the elements $\bb^i$ in the MGS $\mathcal{B}$, so write
 \[\bq = \sum_i f_i \bb^i\]
for some polynomials $f_i \in R^{n+1}$.  Now apply $e^{n+1}$ to get polynomials $g_i = e^{n+1}(f_i)$ in $R^n$.  Consider the image of $\sum_i g_i \bb^i $ under the module homomorphism $e^{n+1}_*$ and expand as follows:
\[e^{n+1}_*\left( \sum_i g_i \bb^i \right) = \sum_i e^{n+1}(g_i) e^{n+1}_*(\bb^i).\] 
Since each $g_i \in R^n$ we have $e^{n+1}(g_i) = g_i$.  Substituting the definition $g_i = e^{n+1}(f_i)$, we get
\[ \sum_i e^{n+1}(g_i) e^{n+1}_*(\bb^i) = \sum_i e^{n+1}(f_i) e^{n+1}_*(\bb^i). \]
But this is simply $e^{n+1}_*(\bq)$, which is $\bp$ by definition.  In other words, every $\bp \in R^n_{G,\alpha}$ is in the image of $R^n(\mathcal{B})$ under $e^{n+1}_*$.

Suppose that $\mathcal{C}$ is an MGS for $R^n_{G,\alpha}$ and denote the homogenization of all of its elements by $\widetilde{\mathcal{C}}$.  By definition $R^n(\mathcal{C})$ is all of $R^n_{G,\alpha}$.  Moreover the map $e^{n+1}_*$ is an isomorphism of $R^n$-modules
\[R^n(\mathcal{C}) \cong R^n(\widetilde{\mathcal{C}}).\]
It is surjective because the image $e^{n+1}_*(R^n(\widetilde{\mathcal{C}}))$ must contain all of $R^n(\mathcal{C})$.  We deduce injectivity as follows.  Suppose $\sum f_i \widetilde{\bp_i}$ is an element of the kernel when $x_{n+1}$ is evaluated at $1$.  Then $\sum f_i \bp_i = 0$ is a relation satisfied by the original elements ${\bp_i} \in {\mathcal{C}}$.

Now consider the quotient $R^{n+1}_{G,\widetilde{\alpha}}/\ker e^{n+1}_*$.  By the first isomorphism theorem, it is isomorphic to the image $R^n_{G,\alpha}$.  We just confirmed that $R^n{(\widetilde{\mathcal{C}})}$ gives a complete set of coset representatives.  So $R^n{(\widetilde{\mathcal{C}})}$ generates the quotient.  

We claim that $R^{n+1}(\widetilde{\mathcal{C}})$ is all of $R^{n+1}_{G,\widetilde{\alpha}}$.  Suppose $\bq \in \ker e^{n+1}_*$ is a minimal-degree element not generated by the $\widetilde{\mathcal{C}}$.  Then $\bq = (x_{n+1}-1)\bq_1$ for some $\bq_1 \in R^{n+1}_{G,\widetilde{\alpha}}$ of strictly smaller degree than $\bq$.  Since $\bq_1 \in R^{n+1}_{G,\widetilde{\alpha}}$ we may multiply by $(x_{n+1}-1)$ to conclude that $\bq$ is, too.

So $\widetilde{\mathcal{C}}$ is a homogeneous generating set for $R^{n+1}_{G,\widetilde{\alpha}}$.  If $\mathcal{B}$ had smaller cardinality than $\mathcal{C}$ then $e^{n+1}_*(\mathcal{B})$ would be a smaller generating set for $R^n_{G,\alpha}$ than $\mathcal{C}$, which contradicts the definition of MGS.  But $\mathcal{B}$ can not be larger than $\widetilde{\mathcal{C}}$ since $\mathcal{B}$ is an MGS.  Thus any MGS of $R^n_{G,\alpha}$ has the same cardinality as a homogeneous MGS for $R^{n+1}_{G,\widetilde{\alpha}}$ and so $e^{n+1}_*$ must be injective on $R^n(\mathcal{B})$.

Finally suppose that for each edge $uv$ the ideal $\alpha(uv)$ is principal.  We note first that if $\bp \in R^{n+1}_{G,\widetilde{\alpha}}$ and $x_{n+1}|\bp_u$ for all vertices $u$, then in fact $x_{n+1}|\bp$.
Indeed, let $uv$ be an edge and suppose $\alpha(uv)$ is generated by the polynomial $p_{uv} \in R^n$.  Since $\bp$ is a spline we have
\[\bp_u - \bp_v = f \widetilde{p_{uv}}\]
for some polynomial $f \in R^{n+1}$ and since $x_{n+1}$ divides each $\bp_u$ we know that $x_{n+1}$ divides the righthand side of this equation.  The definition of homogenization implies that at least one term in $\widetilde{p_{uv}}$ is not divisible by $x_{n+1}$.  The coefficient ring $R^{n+1}$ is an integral domain so the degree of every product is additive.  In particular, all terms in $f$ must be divisible by $x_{n+1}$ in order to guarantee that all terms in the product $f \widetilde{p_{uv}}$ are divisible by $x_{n+1}$.  Thus $x_{n+1} | f$.  We thus have polynomials $\bq_u$ and $g$ with $f = x_{n+1}g$ and $\bp_u = x_{n+1} \bq_u$ for all $u$ and furthermore have
\[\bq_u - \bq_v = g \widetilde{p_{uv}},\]
so $\bq \in R^{n+1}_{G,\widetilde{\alpha}}$ is a spline.

We may thus assume that for each spline $\bb$ in $\mathcal{B}$, there is at least one vertex $u$ such that $\bb_u$ is a nonzero polynomial in $R^n$.  Otherwise, by the previous paragraph, we could write $\bb = x_{n+1} \bq$ for some homogeneous spline $\bq$ and then replace $\bb$ by $\bq$ in the MGS $\mathcal{B}$.  Hence $e^{n+1}_*$ preserves the degree of the splines in $\mathcal{B}$.
\end{proof}

We now describe classical spline theory as it relates to our constructions.

\section{Applications}
\label{sec:applications}

The purpose of this section is to connect the theoretical apparatus of the previous parts of the paper to the spline theory arising in applications.  In particular, we show that our hypothesis on the homogeneity of polynomial splines in Section~\ref{sec:cycles} is a very common one in applications.  In fact, it is satisfied by all graphs arising from GKM constructions of equivariant cohomology and by most applications involving classical splines. We then use our results to recover well-known results that classify classical splines on ``pinwheel" triangulations in the plane.

\subsection{Splines in equivariant cohomology: GKM theory} \label{section: GKM summary}
GKM theory is the name given to an algebraic combinatorial approach to studying torus-equivariant cohomology of certain algebraic varieties by restricting to the torus-fixed points.  We omit details and instead refer to the original source on GKM theory by Goresky, Kottwitz, and MacPherson \cite{GKM98} and to surveys like \cite{T16b} (for a spline-centric approach to GKM theory) or \cite{Tym05}.  We use this subsection only to summarize the key points.

GKM theory gives conditions on algebraic varieties $X$ with an action of a torus $T \cong \left(\mathbb{C}^*\right)^n$ under which we can construct an edge-labeled graph $(G_X,\alpha_X)$ giving an isomorphism between $T$-equivariant cohomology and splines
\[H^*_T(X; \mathbb{C}) \cong R_{G_X, \alpha_X}\]
over the ring $R = \mathbb{C}[t_1,\ldots,t_n]$.  The map is an isomorphism both of rings and of $\mathbb{C}[t_1,\ldots,t_n]$-modules. 

The topological conditions on the torus actions have the following consequence for the graphs that arise in GKM theory.

\begin{proposition} \label{proposition: GKM implies homogeneous}
All edge-labeled graphs $(G_X, \alpha_X)$ that arise from GKM constructions have principal ideals generated by homogeneous polynomials of degree one.
\end{proposition}

In addition, the quotient construction of Section~\ref{sec: quotient splines}, and especially Corollary~\ref{theorem: homogenizing splines}, arises naturally in GKM theory as follows.

\begin{proposition} \label{proposition: GKM and quotient splines}
Suppose that $X$ is an algebraic variety and $T$ is a torus acting on $X$ so that $T$ and $X$ satisfy the topological conditions of GKM theory.  Let $I_1$ denote the (homogeneous) ideal of polynomials in $\mathbb{C}[t_1,\ldots,t_n]$ generated by the variables $t_1, t_2, \ldots, t_n$.  Then the ordinary cohomology of $X$ can be written as
\[H^*(X) \cong H^*_T(X)/I_1H^*_T(X).\]
\end{proposition}

\subsection{Classical results on splines}\label{section: classical spline summary}
The remainder of this article deals with classical splines.  Classical splines are defined as piecewise polynomials on a particular form of geometric decomposition of a space (triangulation, polyhedral, etc.), usually restricted to degree at most $d$ and differentiability at least $r$.  For our purposes, it is sufficient to take $\Delta$ to be a finite $n$-dimensional simplicial complex embedded in $\R^n$ with set of $n$-dimensional simplices $\{\sigma_v\}$.  We view $\Delta$ as both an abstract set of simplices and as a subset of $\R^n$, depending on context.
\begin{definition} \label{definition: classical splines}
Let $r$ and $d$ be nonnegative integers. The \emph{space of splines} $S^r_d(\Delta)$ is the $\mathbb{R}$-vector space defined by the property that $F \in S^r_d(\Delta)$ if and only if $F\colon \Delta \rightarrow \R$ is a function that
\begin{itemize}
    \item has degree at most $d$ in the sense that each restriction $F|_{\sigma_v}$ is a polynomial of degree at most $d$, and
    \item is continuously differentiable of order $r$ as a function defined on a subspace of $\R^n$, namely $F$ is in $\mathcal{C}^r$.
\end{itemize} 
\end{definition}

The splines we consider elsewhere in this paper are a dualization of the classical splines defined in Definition~\ref{definition: classical splines}, as follows.

\begin{definition} \label{definition: dual and dualizing map}
Suppose $\Delta$ is an $n$-dimensional simplicial complex with $n$-simplices $\{\sigma_v\}$.

\begin{itemize}
\item The {\em dual graph} $G_{\Delta}$ is the graph whose vertex set $V_{\Delta}$ is indexed by the collection of $n$-dimensional simplices $\sigma_v \in \Delta$ and whose edge set $E_{\Delta}$ contains the edge $uv$ whenever the corresponding $n$-simplices intersect in an $(n-1)$-dimensional simplex $\sigma_u \cap \sigma_v$.  
\item The {\em dual edge-labeling} $\alpha_{\Delta}$ is the edge-labeling in which $uv$ is labeled by the principal ideal $\alpha_{\Delta}(uv)$  generated by any nonzero affine linear form in $\mathbb{R}[x_1,\ldots,x_n]$ that vanishes on $\sigma_u \cap \sigma_v$. 
\item The {\em dual map} is a map from elements of $S^r_d(\Delta)$ to functions on the vertex set $V_{\Delta}$ of $G_{\Delta}$.  If $F \in S^r_d(\Delta)$, then the dual map sends $F$ to the function $F^*\co V_{\Delta} \rightarrow \mathbb{R}[x_1, \ldots, x_n]$ defined by $F^*(v) = F|_{\sigma_v}$ for all vertices $v$.
\end{itemize}
\end{definition}

Note that if $\Delta$ is a triangulation embedded in the plane, we may consider its vertices and edges as a planar graph, in which case $G_{\Delta}$ is the usual (combinatorial) dual graph to $\Delta$.  In this case, the dual edge-labeling $\alpha_{\Delta}$ is the function that assigns to each edge $uv$ in $G_{\Delta}$ the ideal generated by the equation of the line at the intersection of the triangles corresponding to $u$ and $v$. We note that the edge-labeling of Figure \ref{figure: example of non-homogenized graph} cannot occur in a graph $G_{\Delta}$ dual to a standard triangulation $\Delta$ because all edge-labels $\alpha_{\Delta}(uv)$ have the form $(ax+by+c)^{r+1}$ for some $a, b, c \in \mathbb{R}$.

Billera proved that the dual map is actually an isomorphism of vector spaces between classical splines and splines as defined in Definition~\ref{def:generalizedspline}.  In other words, the splines used in this paper are a kind of dualization of classical splines. We describe Billera's result in Proposition~\ref{proposition: billera result} below.  We do not give a formal definition of {\em simplex}, {\em strongly connected}, or {\em link} and instead refer the interested reader to either \cite{Bil88} or any introductory text on polytopes (see, for example, \cite{Zie95}). Recall also that for principal ideals, a power of an ideal is equal to the ideal generated by that power of the generator.

\begin{proposition}[{Billera \cite[Theorem 2.4]{Bil88}}] \label{proposition: billera result}
Suppose $\Delta$ is a strongly-connected $n$-dimensional simplicial complex so that the link of each simplex in $\Delta$ is also strongly connected.  

Define the $(r+1)^{\text{th}}$-power of $\alpha_\Delta$ to be the edge-labeling $\alpha_{\Delta}^{r+1}$  that associates to each edge $uv$ the ideal $\left( \alpha_{\Delta}(uv) \right)^{r+1}$. Consider the module of splines $R_{G_{\Delta},\alpha^{r+1}_{\Delta}}$ with coefficients in the polynomial ring $R = \mathbb{R}[x_1,\ldots,x_n]$.

Then $F \in S^r_{d}(\Delta)$ if and only if $F^* \in R_{G_{\Delta}, \alpha^{r+1}_{\Delta}}$ is a spline whose localizations $F^*(v)$  have degree at most $d$ for all $v \in V_{\Delta}$.  As $\mathbb{R}$-vector spaces, this dual map is an isomorphism.
\end{proposition}

The hypotheses in the first sentence of Proposition~\ref{proposition: billera result} are  satisfied by most decompositions that arise in applications.  For instance, suppose the simplicial complex $\Delta$ is a triangulation of a region in the plane. In this case, that the link of a vertex is strongly connected means that there are no ``pinch points" in the region; that is, the region's boundary is a disjoint union of subspaces homotopic to circles.  

Splines are strictly more general than classical splines in various ways: dual graphs to triangulations must have trivalent interior vertices (or other regularity conditions on most vertices, in the case of more general simplices); indeed, dual graphs to planar graphs are planar.  On the algebraic side, the ideals that arise in the image of $\alpha_{\Delta}$  must be principal (unlike most ideals), and the underlying ring is a polynomial ring or quotient thereof (unlike most rings). 

We now specialize Corollary \ref{corollary: splines over quotients} to splines on dual graphs, using Billera's result from Proposition \ref{proposition: billera result} to show the vector space of classical splines $S^r_d(\Delta)$ is the space of splines $(R/I_{d+1})_{G_{\Delta},\pi \circ ({\alpha_{\Delta}})^{r+1}}$ over a quotient polynomial ring.  Note that this endows $S^r_d(\Delta)$ with a product structure.

\begin{corollary} \label{corollary: homogeneity means quotient works}
Assume $\Delta$ satisfies the hypotheses of Proposition~\ref{proposition: billera result} and let $R = \mathbb{R}[x_1,\ldots,x_n]$.  Suppose that every ideal $\alpha(uv)$ is a principal ideal generated by a homogeneous element in $R$, let $I_{d+1}$ be the ideal in $R$ generated by all homogeneous elements of degree $d+1$ and let $\pi\co R \rightarrow R/I_{d+1}$ be the quotient map.  Let $(R/I_{d+1})_{G_{\Delta}, \pi \circ ({\alpha_{\Delta}})^{r+1}}$ denote the ring of splines on the edge-labeled dual graph over the quotient ring $R/I_{d+1}$.

Define $\varphi\co S^r_{d}(\Delta) \rightarrow (R/I_{d+1})_{G_{\Delta}, \pi\circ ({\alpha_{\Delta}})^{r+1}}$ to be the dual map 
\[S^r_d(\Delta) \rightarrow \mathcal{S}_d(G_{\Delta},\alpha_{\Delta}^{r+1})\] 
that sends a classical spline $F$ to the spline $F^*$ composed with the map 
\[\pi_*\co \mathcal{S}_d(G_{\Delta},\alpha_{\Delta}^{r+1}) \rightarrow (R/I_{d+1})_{G_{\Delta}, \pi \circ ({\alpha_{\Delta}})^{r+1}}\] 
defined in Proposition~\ref{proposition: quotient splines and reducing coeffs}, so that $\varphi(F)=\pi_*(F^*)$.

Then $F \in S^r_{d}(\Delta)$ if and only if $\varphi(F) \in (R/I_{d+1})_{G_{\Delta}, \pi\circ ({\alpha_{\Delta}})^{r+1}}$.  Moreover the map $\varphi$ is an $\mathbb{R}$-vector space isomorphism and respects multiplication in the following senses: 
\begin{itemize}
\item 	If $f \in \mathbb{R}[x_1, \ldots, x_n]$ and $F \in S^r_d(\Delta)$ satisfy $f F \in S^r_d(\Delta)$, then 
\[\varphi(f F)= f \varphi(F) = \pi(f)\varphi(F) \in (R/I_{d+1})_{G, \pi \circ (\alpha_{\Delta})^{r+1}}.\] 
\item If $F_1, F_2 \in S^r_d(\Delta)$ satisfy $F_1F_2 \in S^r_d(\Delta)$, then 
 \[\varphi(F_1 F_2) = \varphi(F_1) \varphi(F_2) \in (R/I_{d+1})_{G, \pi \circ ({\alpha_{\Delta}})^{r+1}}.\]
\end{itemize}
\end{corollary}

\begin{proof}
Billera's original result proved that the dual map sending $F \mapsto F^*$ is a well-defined isomorphism of real vector spaces.  Corollary \ref{corollary: splines over quotients} composes with the quotient map and completes the proof.
\end{proof}

The main tool in the final section is the following lemma, which establishes that generators for edge-labeled graphs $(G,\alpha)$ can be assumed to be in a certain kind of general position.  More importantly, the lemma constrains the edge-labeled graphs $(G,\alpha)$ that can arise as the duals to a simplicial complex.  In essence, it says that we may change coordinates to assume that any particular interior vertex in a simplicial complex is the origin of the plane, and then reinterprets that for the edge-labeling of the dual graph.

This lemma reinforces the point that generalized splines are more general than splines dual to simplicial complexes.  It is not generally true that every edge-labeling can be modified by an affine linear operator so that any individual vertex is incident only to ideals generated by polynomials with nonconstant terms!  This is a constraint on the graph inherited from the geometry of graphs embedded in Euclidean space.

\begin{lemma} \label{lemma: assume no edge labeled $y$}
Let $R$ denote the polynomial ring $\Bbbk[x_1,\ldots,x_n]$ where $\Bbbk$ is a field of characteristic zero, especially $\mathbb{R}$ or $\mathbb{C}$.  Suppose that $(G, \alpha)$ is a graph for which each edge-label $\alpha(uv)$ is the principal ideal generated by some power of an affine form, namely
\[(a_{1,uv} x_1+a_{2,uv} x_2+\cdots + a_{n,uv} x_n+c_{uv})^{r_{uv}}\] 
for some positive integer $r_{uv}$ and constants $a_{1,uv}, \ldots, a_{n,uv}, c_{uv} \in \Bbbk$. 

For each vertex $v_0$, there is an edge-labeling $\alpha'$ so that 

\begin{enumerate}
    \item all edge-labels $\alpha'(uv)$ are generated by polynomials with every coefficient $a'_{i,uv}$ nonzero, 
    \item $\alpha'$ is constructed by composing $\alpha$ with a linear operator that acts as a rotation, and
    \item \[R_{G,\alpha} \cong R_{G,\alpha'}.\]
    \end{enumerate}

Furthermore, suppose $\Delta$ is any triangulation in the plane that satisfies the hypotheses of Proposition~\ref{proposition: billera result}, with dual graph $G_{\Delta}$ and dual edge-labeling $\alpha_{\Delta}^{r+1}$.  Fix a bounded face $C$ in $G_{\Delta}$.  Then the edge-labeling $\alpha'$ can also be chosen to satisfy

\begin{enumerate}
    \item[(4)] all edge-labels $\alpha'(uv)$ bounding $C$ are generated by homogeneous polynomials, and
    \item[(5)] $\alpha'$ can be constructed by composing $\alpha_{\Delta}$ with the linear operator that translates the vertex $C \in \Delta$ to the origin and then performs a rotation of Euclidean space around the origin.  
\end{enumerate}
\end{lemma}

\begin{proof}
Choose a vector $\vec{p} \in \Bbbk^n$ and an invertible $n \times n$ matrix $A \in GL_n(\Bbbk)$ and denote the entries of $A$ by $A_{ij}$.  Define maps $\varphi_A$ and $\varphi_{\vec{p}}$ on the variables $x_1, \ldots, x_n$ by the rules
\[\varphi_A(x_i) = \sum_{j=1}^n A_{ji}x_j \hspace{0.25in} \textup{ and } \hspace{0.25in} \varphi_{\vec{p}}(x_i) = x_i + p_i\]
and extend this to a ring homomorphism on all of $\Bbbk[x_1,\ldots,x_n]$.  The map $\varphi_A$ can be thought of as a change of variables between $x_i$ and $y_i:= \varphi_A(x_i)$ and similarly for $\varphi_{\vec{p}}$.  In particular, we obtain ring isomorphisms
\[\varphi_A\co \Bbbk[x_1,\ldots, x_n] \rightarrow \Bbbk[y_1,\ldots,y_n] \hspace{0.25in} \textup{ and } \hspace{0.25in} \varphi_{\vec{p}}\co \Bbbk[x_1,\ldots, x_n] \rightarrow \Bbbk[y_1,\ldots,y_n] \]
with inverse maps induced from $A^{-1}$ and $-\vec{p}$ respectively.

Now suppose that $(G,\alpha)$ is any edge-labeled graph for which each edge-label $\alpha(uv)$ is generated by some power of an affine form, as per the hypothesis.  Using Proposition~\ref{proposition: surjective mostly functoriality} we obtain maps $(\varphi_A)_*$ and $(\varphi_{\vec{p}})_*$ that are isomorphisms on the corresponding rings  of splines.  Moreover suppose we write $\vec{ax}_{uv} = (a_{1,uv}x_1, a_{2,uv}x_2,\ldots,a_{n,uv}x_n)$ for each edge $uv$, with $a_{i,uv} \in \Bbbk$ and $x_i \in R$. Then the ideal $(\varphi_A \circ \alpha)(uv)$ is generated by 
\[(A \vec{ax}_{uv} + c_{uv})^{r_{uv}}.\]
We now find an invertible matrix $A$ so that for all edges $uv$ the expression $A\vec{ax}_{uv}$ has $n$ nonzero terms.  Indeed, consider the $|E|$ hyperplanes 
\[a_{1,uv} x_1+a_{2,uv} x_2+\cdots + a_{n,uv} x_n = 0\]
obtained over all edges $uv$.  This is a finite set of hyperplanes and $GL_n(\Bbbk)$ consists of an infinite number of matrices. In particular, consider the subgroup of rotations around the origin.  This is a unitary group and so intersects each hyperplane in a subspace of codimension one.  No finite union of these codimension-one subspaces can cover the entire space of possible rotations. Thus there exists a rotation $A$ that satisfies Condition (1) from the claim, as desired.

In addition, suppose $(G_{\Delta},\alpha_{\Delta})$ arises as the dual of a triangulation $\Delta$ satisfying the hypotheses of Proposition~\ref{proposition: billera result}.  Each vertex $\vec{C} \in \Bbbk^2$ of the simplicial complex $\Delta$ satisfies the equations
\[\sum_{i=1}^2 a_{i,uv}x_i + c_{uv} = 0\]
of each line segment through $\vec{C} \in \Delta$.  The translation $\varphi_{-\vec{C}}$ moves $\vec{C}$ to the origin in the plane.  Thus the translation induces a ring homomorphism for which the equations $\alpha'(uv)$ have no constant term whenever $uv$ is an edge bounding the face corresponding to $\vec{C}$ in the dual graph $G_{\Delta}$.  Thus $\alpha'(uv)$ is homogeneous for all edges $uv$ bounding the face corresponding to $\vec{C}$.  

Composing these maps as necessary proves the claim.
\end{proof}

\subsection{Applications to splines on planar triangulations}\label{sec:apps}
In this final subsection, we apply earlier work in this paper to the lower bound formula described in the introduction.  

Throughout this subsection, we denote by $\Delta$ a triangulation of a region in the plane $\mathbb{R}^2$ satisfying the hypotheses of Proposition~\ref{proposition: billera result}.  For all undefined terms and for a detailed history of the lower bound formula and lower bound conjecture, we point to \cite[Section 9]{LaiSch07}. We now explicitly state what's often called Schumaker's lower bound formula.  

\begin{theorem}[Schumaker's lower bound formula {\cite[Theorem 9.9]{LaiSch07}}]\label{thm:lowerbound}
Let $\mathcal{V}_\mathrm{int}$ be the set of interior vertices of $\Delta$.  For each $v \in \mathcal{V}_\mathrm{int}$, denote by $m_v$ the number of distinct slopes among all edges incident to $v$.  Write $V_\mathrm{int} := |\mathcal{V}_\mathrm{int}|$, and similarly denote by $E_\mathrm{int}$ the number of interior edges of $\Delta$.  Then for all $0 \le r \le d$, we have
\[
D + \sigma \le \dim S_d^r(\Delta),
\]
where
\[
D := \binom{d+2}{2} + \binom{d-r+1}{2}E_{\mathrm{int}} - \left[\binom{d+2}{2} - \binom{r+2}{2}\right]V_{\mathrm{int}},
\]
and
\[
\sigma := \sum_{v \in \mathcal{V}_\mathrm{int}} \sigma_v,
\]
where 
\[
\sigma_v := \sum_{j=1}^{d-r} (r+j+1 - jm_v)_+.
\]
Here, the subscript $+$ denotes the function $(-)_+\co \mathbb{R} \rightarrow \mathbb{R}_{\ge 0}$ defined by $x \mapsto x$ if $x >0$ and $x \mapsto 0$ otherwise.
\end{theorem}

This lower bound for the dimension of the space of classical splines $S^r_d(\Delta)$ holds for all $r$, $d$, and $\Delta$ and one active area of research looks for tighter upper- and lower-bounds for $S^r_d(\Delta)$ in general.  We focus instead on another active open question: the case when $\Delta$ is a planar triangulation and the smoothness $r$ is one.  Indeed, our focus on quadratic edge-labels and low-degree splines in parts of this paper is precisely because of the lower bound conjecture. 

\begin{conjecture}
The inequality in the lower bound formula of Theorem~\ref{thm:lowerbound} is an  equality in the case that $r = 1$ and $d \geq 3$; that is, in the notation of Theorem~\ref{thm:lowerbound},
\[
D + \sigma = \dim S^1_d(\Delta).
\]
\end{conjecture}

Some authors call this Schumaker's conjecture, though he credits Strang.  Alfeld--Schumaker and Hong proved it when $d \geq 4$ \cite{AlfSch90,Hon91} and Billera proved it for \textit{generic} triangulations when $d=3$ \cite{Bil88}.  (The case when $d=2$ is so mysterious that it has no conjectural formula.)

For most of the remainder of the paper, we will be focused on the special case whenever $\Delta$ is an interior cell (or {\em pinwheel} triangulation), which is a triangulation that has a unique interior vertex around which triangles radiate like the spokes of a wheel. This is shown in Figure \ref{figure: pinwheel and dual cycle} together with its dual graph, which is a cycle. 
\begin{figure}
    \centering
    \begin{picture}(200,100)(-100,-50)
    \put(-30,42){\circle*{5}}
    \put(30,42){\circle*{5}}
    \put(-30,-42){\circle*{5}}
    \put(30,-42){\circle*{5}}
    \put(-60,0){\circle*{5}}
    \put(60,0){\circle*{5}}
    \put(0,0){\circle*{5}}
    
    \put(-31,-42){\line(3,4){61}}
    \put(-31,42){\line(3,-4){61}}
    \put(-60,0){\line(1,0){120}}
    
    \multiput(-60,0)(90,42){2}{\line(3,-4){30}}
    \multiput(-60,0)(90,-42){2}{\line(3,4){30}}
    \multiput(-30,42)(0,-84){2}{\line(1,0){60}}
    
    \put(0,25){\color{red} \circle*{5}}
    \put(0,-25){\color{red} \circle*{5}}
    \put(-25,12){\color{red} \circle*{5}}
    \put(25,12){\color{red} \circle*{5}}
    \put(-25,-12){\color{red} \circle*{5}}
    \put(25,-12){\color{red} \circle*{5}}

    \multiput(-25,-12)(25,37){2}{\color{red} \line(2,-1){25}}
    \multiput(-25,12)(25,-37){2}{\color{red} \line(2,1){25}}
    \multiput(-25,-12)(50,0){2}{\color{red} \line(0,1){24}}

    \end{picture}
    \caption{A pinwheel with its dual graph shown in red.}
    \label{figure: pinwheel and dual cycle}
\end{figure}
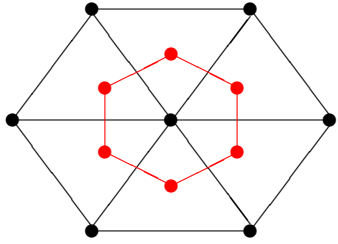

Note that in the case of interior cells $\Delta$, the lower bound formula of Theorem~\ref{thm:lowerbound} reduces to the formula below, and it actually computes the dimension of the space $S^r_d(\Delta)$.

\begin{theorem}[{\cite[Theorem 9.3]{LaiSch07}}]\label{thm:laisch}
Let $\Delta$ be an interior cell (or pinwheel triangulation) with single interior vertex $v$.  For any $0 \le r \le d$, we have
\[
\binom{r+2}{2} +  \binom{d-r+1}{2}E_{\mathrm{int}} + \sigma_v = \dim S^r_d(\Delta).
\]
\end{theorem}

\begin{example}\label{ex:interiorcelldim}
We illustrate Theorem~\ref{thm:laisch} for the two interior cells $\Delta_{\mathrm{left}}$ and $\Delta_{\mathrm{right}}$ of Figure~\ref{figure: two interior cells} in the case $r=1$ and $d=3$.

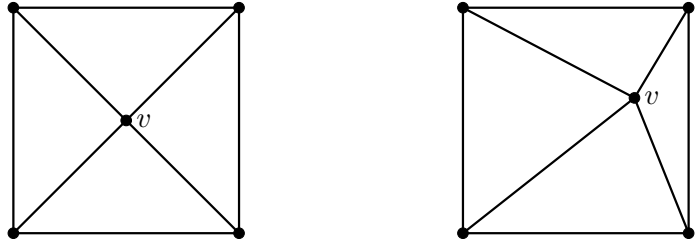
\begin{figure}
    \centering
    \begin{tikzpicture}
       \draw[line width = .85pt] (-1.5,1.5) -- (1.5,-1.5);
\draw[line width=.85pt] (-1.5,-1.5) -- (1.5,1.5);
\draw[line width=.85pt] (-1.5,1.5) -- (-1.5,-1.5);
\draw[line width = .85pt] (-1.5,-1.5) -- (1.5,-1.5);
\draw[line width = .85pt] (1.5,-1.5) -- (1.5,1.5);
\draw[line width = .85pt] (-1.5,1.5) -- (1.5,1.5);
\filldraw[black] (0,0) circle (2pt) node[anchor=west]{$v$};
\filldraw[black] (-1.5,1.5) circle (2pt);
\filldraw[black] (-1.5,-1.5) circle (2pt);
\filldraw[black] (1.5,1.5) circle (2pt);
\filldraw[black] (1.5,-1.5) circle (2pt);
    \end{tikzpicture}
    \qquad \qquad \qquad \quad
        \begin{tikzpicture}
       \draw[line width = .85pt] (-1.5,1.5) -- (1.5,1.5);
\draw[line width=.85pt] (-1.5,1.5) -- (-1.5,-1.5);
\draw[line width=.85pt] (-1.5,-1.5) -- (1.5,-1.5);
\draw[line width = .85pt] (1.5,-1.5) -- (1.5,1.5);

\draw[line width = .85pt] (-1.5,-1.5) -- (.78,.30);
\draw[line width = .85pt] (-1.5,1.5) -- (.78,.30);
\draw[line width = .85pt] (1.5,1.5) -- (.78,.30);
\draw[line width = .85pt] (1.5,-1.5) -- (.78,.30);

\filldraw[black] (.78,.30) circle (2pt) node[anchor=west]{$v$};

\filldraw[black] (-1.5,1.5) circle (2pt);
\filldraw[black] (-1.5,-1.5) circle (2pt);
\filldraw[black] (1.5,1.5) circle (2pt);
\filldraw[black] (1.5,-1.5) circle (2pt);
    \end{tikzpicture}
    \caption{On the left, an interior cell consisting of a singular vertex.  On the right, an interior cell consisting of a vertex that is not singular. Both dual graphs are $4$-cycles but the example on the left has fewer distinct edge-labels.}
    \label{figure: two interior cells}
\end{figure}

Note that both $\Delta_{\mathrm{left}}$ and $\Delta_{\mathrm{right}}$ have $E_\mathrm{int} = 4$.  However, for $\Delta_{\mathrm{left}}$ we have $m_v = 2$, and for $\Delta_{\mathrm{right}}$ we have $m_v = 4$.  Substituting these values into the formula in Theorem~\ref{thm:laisch} gives
\[
\dim S^1_3(\Delta_{\mathrm{left}}) = 16 \qquad \text{and} \qquad \dim S^1_3(\Delta_{\mathrm{right}}) = 15.
\]
\end{example}

One of our main results, Corollary~\ref{corollary: dimension of splines on pinwheels} below, uses our results about minimal generating sets from earlier sections of this paper to recover Theorem~\ref{thm:laisch} of Lai--Schumaker in the case $r=1$.  We first establish some preliminary results.

\begin{theorem}\label{theorem: pinwheel characterization}
Let $\Delta$ be an interior cell, i.e. a pinwheel triangulation, with $n$ triangles. Let $I_{d+1} \subseteq \mathbb{R}[x,y]$ be the ideal generated by all monomials of degree $d+1$, and let $\pi\co R \rightarrow R/I_{d+1}$ be the quotient map.  Then there is an isomorphism of $\mathbb{R}$-vector spaces  $S^r_d(\Delta) \cong (R/I_{d+1})_{C_n,\pi\circ \alpha}$, 
where $C_n$ is an $n$-cycle and $\alpha$ is an edge-labeling so that every ideal $\alpha(uv)$ is principal and generated by $(x+a_{uv}y)^{r+1}$ for a nonzero $a_{uv} \in \mathbb{R}$.  Moreover, for each $a_{uv}$ there is at most one other edge $u'v'$ with $a_{u'v'}=a_{uv}$ and that edge cannot immediately follow or precede $uv$.
\end{theorem}

\begin{proof}
Consider the map $\varphi$ of Corollary~\ref{corollary: homogeneity means quotient works}.  Corollary~\ref{corollary: homogeneity means quotient works} proves that $\varphi$ is an isomorphism of $\mathbb{R}$-vector spaces, and notes that the dual graph to a pinwheel triangulation with $n$ triangles is a cycle on $n$ vertices, with edge-labeling $\alpha$ is given by the $(r+1)^{\text{th}}$ power of the equations of the lines through the central vertex in $\Delta$.

By Lemma~\ref{lemma: assume no edge labeled $y$}, we may translate the central vertex of the triangulation to the origin and assume each edge $uv$ is labeled by $(x+a_{uv}y)^{r+1}$ for nonzero coefficients $a_{uv}$.  Finally, at most two rays through a given point lie on the same line, so no more than two of the edge-labels can coincide; if two successive rays going clockwise around the central vertex are the same, then they describe whose interior angle-sum is more than $180^{\circ}$, which is impossible.  This proves the claim.
\end{proof}

We use the previous theorem to reinterpret the main results of earlier sections.  The key observation is the following, which characterizes cycles that can be realized as the dual of a triangulation.

\begin{lemma}\label{lem:dualcycles}
All edge-labeled cycles $(C,\alpha)$ that are geometrically realizable as the dual of a triangulation must have at least three edge-labels unless the cycle is a four-cycle with two distinct edge-labels that alternate around the cycle.
\end{lemma}

\begin{proof}
A cycle is dual to a triangulation only if that triangulation is an interior cell (namely pinwheel triangulation).  Suppose $(C,\alpha)$ is dual to a triangulation.  Theorem \ref{theorem: pinwheel characterization} implies that if $C$ has three edges, then they are all labeled distinctly; if $C$ has at least five edges, then at least three successive edges must be labeled distinctly.  The only four-cycles with fewer than three distinct edge-labels are precisely those dual to pinwheel triangulations formed by the intersection of two lines.  This gives a four-cycle whose edge-labels alternate between two distinct edge-labels as one moves clockwise.
\end{proof}

Classically, interior vertices formed by the intersection of two lines play a special role in the theory of classical splines on triangulations.  We give this terminology in the context of splines on the dual graph.

\begin{definition}
The interior vertices of the triangulations corresponding to $4$-cycles with two distinct edge-labels are called {\em singular vertices}.  
\end{definition}

We note that the interior vertex $v$ on the left side in Figure~\ref{figure: two interior cells} is a singular vertex, but the vertex $v$ on the right side is not singular.  Remark~\ref{rem:sing} discussed this in the context of the corresponding dual edge-labeled graph.

Lemma \ref{lem:dualcycles} thus shows that singular vertices are special insofar as they correspond to the {\em only geometrically realizable} edge-labeled cycles with two distinct edge labels. 

Combining these results with those from earlier sections gives an explicit algorithm for constructing a minimal generating set for splines on interior cells. The first consequence is a classical result for general $r$ and $d$ \cite[Theorems 9.3 and 9.12]{LaiSch07}.

\begin{corollary} \label{corollary: dimension of splines on pinwheels}
Denote the number of monomials of degree at most $d$ by $m_d$, namely 
\[m_d = 1+\cdots+(d+1) = \frac{(d+1)(d+2)}{2}.\] 
The dimension of the space $S^1_d(\Delta)$ of classical splines on a pinwheel triangulation $\Delta$ (i.e. an interior cell) with $n$ triangles has two formulas.  

If the pinwheel has four triangles and a singular vertex, then the dimension of $S^1_d(\Delta)$ is
\[\begin{cases} m_d & \textup{  if } d \leq 1, \\ m_d + 2m_{d-2} &\textup{  if } 2 \leq d \leq 3, \\ m_d + 2m_{d-2} + m_{d-4} &\textup{   if } d \geq 4. \end{cases}\]
If the pinwheel has $n \geq 3$ triangles and no singular vertex, then the dimension of $S^1_d(\Delta)$ is
\[\begin{cases} m_d & \textup{  if } d \leq 1, \\ m_d + (n-3)m_{d-2} &\textup{  if } d=2, \\ m_d + (n-3)m_{d-2} + 2m_{d-3} &\textup{   if } d \geq 3. \end{cases}\]
\end{corollary}

\begin{proof}
Lemma \ref{lem:dualcycles} asserts that the pinwheel with a  singular vertex is the only geometrically-realizable cycle with just two distinct edge-labels.  Theorem \ref{thm:2labalg} built an upper-triangular basis for the module of splines over the polynomial ring in the two-label case. Each module generator of degree $j$ contributes $m_{d-j}$ elements to the vector space basis in degree at most $d$, since each module generator can be multiplied by each of the $m_{d-j}$ monomials of degree at most $d-j$.

If $\Delta$ is not the pinwheel with a singular vertex, Theorem \ref{theorem: pinwheel characterization} showed that $\Delta$ must have at least three successive distinct labels.  Lemma \ref{lem:reduced} gave a homogeneous basis for the spline space as a module over the polynomial ring in this case.  Thus each module generator of degree $j$ contributes $m_{d-j}$ elements to the vector space basis in degree at most $d$.  This proves the claim.
\end{proof}

\begin{example}
Consider the interior cells $\Delta_{\mathrm{left}}$ and $\Delta_{\mathrm{right}}$ from Figure~\ref{figure: two interior cells}.  Corollary~\ref{corollary: dimension of splines on pinwheels} computes $\dim S^1_3(\Delta_{\mathrm{left}})$ and $\dim S^1_3(\Delta_{\mathrm{right}})$ in a different way than in Example~\ref{ex:interiorcelldim}.

Indeed, note that $m_3 = 10$, $m_1 = 3$, and $m_0 = 1$.  Substituting these values into the formulas in Corollary~\ref{corollary: dimension of splines on pinwheels}, we again obtain
\[
\dim S^1_3(\Delta_{\mathrm{left}}) = 16 \qquad \text{and} \qquad \dim S^1_3(\Delta_{\mathrm{right}}) = 15.
\]
\end{example}

These results also allow us to contextualize the lower bound conjecture.  In particular, we can bound the dimension of $S^1_d(\Delta)$ by building the triangulation $\Delta$ one interior vertex at a time, and by using Corollary \ref{corollary: dimension of splines on pinwheels} to bound the contribution of each interior vertex.

\begin{corollary}\label{cor:final}
Suppose $\Delta$ and $\Delta'$ are triangulations of a region in the plane satisfying the hypotheses of Proposition~\ref{proposition: billera result} and that $\Delta'$ is obtained by adding a new interior cell to $\Delta$ with $k$ triangles radiating around the new interior vertex. 

Then the complex vector space $S^1_d(\Delta')$ may have more basis elements than $S^1_d(\Delta)$.  The number of additional (vector space) basis elements is at most
\[\dim \left( S^1_d(\Delta_0) \right) - m_d,\]
where $\Delta_0$ is the pinwheel triangulation with $k$ triangles and $m_d$ is the number of monomials of degree at most $d$.
\end{corollary}

\begin{proof}
The preimage of the restriction map $R_{G_{\Delta'},\alpha_{\Delta'}} \rightarrow R_{G_{\Delta},\alpha_{\Delta}}$ consists of the {\em nonconstant} splines in $R_{G_{\Delta_0},\alpha}$.  The dimension of nonconstant splines is an upper bound on the total dimension of $R_{G_{\Delta'},\alpha_{\Delta'}}$ since the restriction might not be surjective.  This dimension was given in Corollary \ref{corollary: dimension of splines on pinwheels}, proving the claim.
\end{proof}

The condition that the link of a vertex is strongly connected in fact implies that any triangulation satisfying the constraints of Proposition \ref{proposition: billera result} can be built one interior vertex at a time.   We sketch the argument here.  Since the link of each vertex is strongly connected, the link of each interior vertex is a cycle.  If $\Delta'$ has an interior vertex, then there is an interior vertex lying on a triangle with a boundary edge. Removing this vertex and the triangles on which it lies leaves a triangulation $\Delta$.  If $\Delta$ is connected, then it still satisfies the conditions of Proposition \ref{proposition: billera result}.  For some choice of vertex $\Delta$ is connected, because $\Delta'$ is strongly connected.  


\bibliography{splines}

\newcommand{\etalchar}[1]{$^{#1}$}
\providecommand{\bysame}{\leavevmode\hbox to3em{\hrulefill}\thinspace}
\providecommand{\MR}{\relax\ifhmode\unskip\space\fi MR }
\providecommand{\MRhref}[2]{%
  \href{http://www.ams.org/mathscinet-getitem?mr=#1}{#2}
}
\providecommand{\href}[2]{#2}
\begin{thebibliography}{ACFG{\etalchar{+}}20}

\bibitem[ACFG{\etalchar{+}}20]{ACFGMT20}
K.~Anders, A.~Crans, B.~Foster-Greenwood, B.~Mellor, and J.~Tymoczko,
  \emph{Graphs admitting only constant splines}, Pacific J. Math. \textbf{304}
  (2020), no.~2, 385--400.

\bibitem[AHM19]{AHM19}
H.~Abe, T.~Horiguchi, and M.~Masuda, \emph{The cohomology rings of regular
  semisimple {H}essenberg varieties for {$h=(h(1),n,\dots,n)$}}, J. Comb.
  \textbf{10} (2019), no.~1, 27--59.

\bibitem[AS87]{AlfSch87}
P.~Alfeld and L.~L. Schumaker, \emph{The dimension of bivariate spline spaces
  of smoothness {$r$} for degree {$d\geq 4r+1$}}, Constr. Approx. \textbf{3}
  (1987), no.~2, 189--197.

\bibitem[AS90]{AlfSch90}
\bysame, \emph{On the dimension of bivariate spline spaces of smoothness {$r$}
  and degree {$d=3r+1$}}, Numer. Math. \textbf{57} (1990), no.~6-7, 651--661.

\bibitem[AS19a]{AltSar19b}
S.~Altinok and S.~Sarioglan, \emph{Basis criteria for generalized spline
  modules via determinant}, arXiv:1903.08968, 2019.

\bibitem[AS19b]{AltSar19}
\bysame, \emph{Flow-up bases for generalized spline modules on arbitrary
  graphs}, arXiv:1902.03756, 2019.

\bibitem[BBB87]{BBB87}
R.~H. Bartels, J.~C. Beatty, and B.~A. Barsky, \emph{An introduction to splines
  for use in computer graphics and geometric modeling}, Morgan Kaufmann, Palo
  Alto, CA, 1987, With forewords by Pierre B\'{e}zier and A. Robin Forrest.

\bibitem[BC18]{BC18}
P.~Brosnan and T.~Y. Chow, \emph{Unit interval orders and the dot action on the
  cohomology of regular semisimple {H}essenberg varieties}, Adv. Math.
  \textbf{329} (2018), 955--1001.

\bibitem[BHKR15]{BHKR15}
N.~Bowden, S.~Hagen, M.~King, and S.~Reinders, \emph{Bases and structure
  constants of generalized splines with integer coefficients on cycles},
  arXiv:1502.00176, 2015.

\bibitem[Bil88]{Bil88}
L.~J. Billera, \emph{Homology of smooth splines: generic triangulations and a
  conjecture of {S}trang}, Trans. Amer. Math. Soc. \textbf{310} (1988), no.~1,
  325--340.

\bibitem[BR91]{BilRos91}
L.~J. Billera and L.~L. Rose, \emph{A dimension series for multivariate
  splines}, Discrete Comput. Geom. \textbf{6} (1991), no.~2, 107--128.

\bibitem[BT15]{BowTym15}
N.~Bowden and J.~Tymoczko, \emph{Splines mod m}, arXiv:1501.02027, 2015.

\bibitem[CLO15]{CLO15}
D.~A. Cox, J.~Little, and D.~O'Shea, \emph{Ideals, varieties, and algorithms},
  fourth ed., Undergraduate Texts in Mathematics, Springer, Cham, 2015, An
  introduction to computational algebraic geometry and commutative algebra.

\bibitem[dB01]{dB01}
C.~de~Boor, \emph{A practical guide to splines}, revised ed., Applied
  Mathematical Sciences, vol.~27, Springer-Verlag, New York, 2001.

\bibitem[DiP12]{DiP12}
M.~R. DiPasquale, \emph{Shellability and freeness of continuous splines}, J.
  Pure Appl. Algebra \textbf{216} (2012), no.~11, 2519--2523.

\bibitem[DS20]{DS20}
M.~DiPasquale and F.~Sottile, \emph{Bivariate semialgebraic splines}, J.
  Approx. Theory \textbf{254} (2020), 105392, 19.

\bibitem[GKM98]{GKM98}
M.~Goresky, R.~Kottwitz, and R.~MacPherson, \emph{Equivariant cohomology,
  {K}oszul duality, and the localization theorem}, Invent. Math. \textbf{131}
  (1998), no.~1, 25--83.

\bibitem[GP13]{GP13}
M.~Guay-Paquet, \emph{A modular relation for the chromatic symmetric functions
  of (3+1)-free posets}, arXiv preprint arXiv:1306.2400 (2013).

\bibitem[GS]{M2}
D.~R. Grayson and M.~E. Stillman, \emph{Macaulay2, a software system for
  research in algebraic geometry}, Available at
  \url{http://www.math.uiuc.edu/Macaulay2/}.

\bibitem[GTV16]{GTV16}
S.~Gilbert, J.~Tymoczko, and S.~Viel, \emph{Generalized splines on arbitrary
  graphs}, Pacific J. Math. \textbf{281} (2016), no.~2, 333--364.

\bibitem[GZ00]{GuiZar00}
V.~Guillemin and C.~Zara, \emph{Equivariant de {R}ham theory and graphs},
  Surveys in differential geometry, Surv. Differ. Geom., vol.~7, Int. Press,
  Somerville, MA, 2000, pp.~221--257.

\bibitem[GZ01a]{GuiZar01d}
\bysame, \emph{1-skeleta, {B}etti numbers, and equivariant cohomology}, Duke
  Math. J. \textbf{107} (2001), no.~2, 283--349.

\bibitem[GZ01b]{GuiZar01}
\bysame, \emph{{$G$}-actions on graphs}, Internat. Math. Res. Notices (2001),
  no.~10, 519--542.

\bibitem[GZ03]{GuiZar03}
\bysame, \emph{The existence of generating families for the cohomology ring of
  a graph}, Adv. Math. \textbf{174} (2003), no.~1, 115--153.

\bibitem[Hon91]{Hon91}
D.~Hong, \emph{Spaces of bivariate spline functions over triangulation},
  Approx. Theory Appl. \textbf{7} (1991), no.~1, 56--75.

\bibitem[HP18]{HP18}
M.~Harada and M.~Precup, \emph{Upper-triangular linear relations on
  multiplicities and the {S}tanley-{S}tembridge conjecture}, arXiv preprint
  arXiv:1812.09503 (2018).

\bibitem[HPT21]{HPT21}
M.~Harada, M.~Precup, and J.~Tymoczko, \emph{Permutation bases in the
  equivariant cohomology rings of regular semisimple {H}essenberg varieties},
  arXiv preprint arXiv:2101.03191 (2021).

\bibitem[HT17]{HarTym17}
M.~Harada and J.~Tymoczko, \emph{Poset pinball, {GKM}-compatible subspaces, and
  {H}essenberg varieties}, J. Math. Soc. Japan \textbf{69} (2017), no.~3,
  945--994.

\bibitem[KT03]{KT03}
A.~Knutson and T.~Tao, \emph{Puzzles and (equivariant) cohomology of
  {G}rassmannians}, Duke Math. J. \textbf{119} (2003), no.~2, 221--260.

\bibitem[LS07]{LaiSch07}
M.-J. Lai and L.~L. Schumaker, \emph{Spline functions on triangulations},
  Encyclopedia of Mathematics and its Applications, vol. 110, Cambridge
  University Press, Cambridge, 2007.

\bibitem[Sch79]{Sch79}
L.~L. Schumaker, \emph{On the dimension of spaces of piecewise polynomials in
  two variables}, Multivariate approximation theory ({P}roc. {C}onf., {M}ath.
  {R}es. {I}nst., {O}berwolfach, 1979), Internat. Ser. Numer. Math., vol.~51,
  Birkh\"{a}user, Basel-Boston, Mass., 1979, pp.~396--412.

\bibitem[SKKT00]{SKKT00}
K.~E. Smith, L.~Kahanp\"{a}\"{a}, P.~Kek\"{a}l\"{a}inen, and W.~Traves,
  \emph{An invitation to algebraic geometry}, Universitext, Springer-Verlag,
  New York, 2000.

\bibitem[SS02]{SchSti02}
H.~K. Schenck and P.~F. Stiller, \emph{Cohomology vanishing and a problem in
  approximation theory}, Manuscripta Math. \textbf{107} (2002), no.~1, 43--58.

\bibitem[SSY19]{SSY}
H.~K. Schenck, M.~E. Stillman, and B.~Yuan, \emph{A new bound for smooth spline
  spaces}, arXiv:1909.13399, 2019.

\bibitem[Str74]{Str74}
G.~Strang, \emph{The dimension of piecewise polynomial spaces, and one-sided
  approximation}, Conference on the {N}umerical {S}olution of {D}ifferential
  {E}quations ({U}niv. {D}undee, {D}undee, 1973), 1974, pp.~144--152. Lecture
  Notes in Math., Vol. 363.

\bibitem[SW16]{SW16}
J.~Shareshian and M.~L. Wachs, \emph{Chromatic quasisymmetric functions}, Adv.
  Math. \textbf{295} (2016), 497--551.

\bibitem[Tym05]{Tym05}
J.~S. Tymoczko, \emph{An introduction to equivariant cohomology and homology,
  following {G}oresky, {K}ottwitz, and {M}ac{P}herson}, Snowbird lectures in
  algebraic geometry, Contemp. Math., vol. 388, Amer. Math. Soc., Providence,
  RI, 2005, pp.~169--188.

\bibitem[Tym16a]{T16}
J.~Tymoczko, \emph{Billey's formula in combinatorics, geometry, and topology},
  Schubert calculus---{O}saka 2012, Adv. Stud. Pure Math., vol.~71, Math. Soc.
  Japan, [Tokyo], 2016, pp.~499--518.

\bibitem[Tym16b]{T16b}
\bysame, \emph{Splines in geometry and topology}, Comput. Aided Geom. Design
  \textbf{45} (2016), 32--47.

\bibitem[YS19]{StiYua19}
B.~Yuan and M.~E. Stillman, \emph{A counter-example to the {S}chenck-{S}tiller
  ``{$2r+1$}'' conjecture}, Adv. in Appl. Math. \textbf{110} (2019), 33--41.

\bibitem[Zie95]{Zie95}
G.~M. Ziegler, \emph{Lectures on polytopes}, Graduate Texts in Mathematics,
  vol. 152, Springer-Verlag, New York, 1995.

\end{thebibliography}
\bibliographystyle{amsalpha}
\end{document}